\documentclass{article}%
\usepackage{amsfonts}
\usepackage{amsmath}
\usepackage{amssymb}
\usepackage{graphicx}%
\newtheorem{theorem}{Theorem}

\newtheorem{definition}[theorem]{Definition}

\newtheorem{lemma}[theorem]{Lemma}

\newtheorem{remark}[theorem]{Remark}

\newenvironment{proof}[1][Proof]{\noindent\textbf{#1.} }{\ \rule{0.5em}{0.5em}}
\begin{document}

\title{A general stochastic maximum principle for optimal control problems of
forward-backward systems}
\author{\textbf{Seid BAHLALI\thanks{Laboratory of Applied Mathematics, University Med
Khider, Po. Box 145, Biskra 07000, Algeria. sbahlali@yahoo.fr \ }}}
\maketitle

\begin{abstract}
Stochastic maximum principle of nonlinear controlled forward-backward systems,
where the set of strict (classical) controls need not be convex and the
diffusion coefficient depends explicitly on the variable control, is an open
problem impossible to solve by the classical method of spike variation. In
this paper, we introduce a new approach to solve this open problem and we
establish necessary as well as sufficient conditions of optimality, in the
form of global stochastic maximum principle, for two models. The first
concerns the relaxed controls, who are a measure-valued processes. The second
is a restriction of the first to strict control problems.

\ 


\ 


\end{abstract}

\section{Introduction}

We study a stochastic control problem where the system is governed by a
nonlinear forward-backward stochastic differential equation (FBSDE\ for short)
of the type%
\[
\left\{
\begin{array}
[c]{l}%
dx_{t}^{v}=b\left(  t,x_{t}^{v},v_{t}\right)  dt+\sigma\left(  t,x_{t}%
^{v},v_{t}\right)  dW_{t},\\
x_{0}^{v}=x,\\
dy_{t}^{v}=-f\left(  t,x_{t}^{v},y_{t}^{v},z_{t}^{v},v_{t}\right)
dt+z_{t}^{v}dW_{t},\\
y_{T}^{v}=\varphi\left(  x_{T}^{v}\right)  ,
\end{array}
\right.
\]
where $b,$ $\sigma,\ f$ and $\varphi$ are given maps, $W=\left(  W_{t}\right)
_{t\geq0}$ is a standard Brownian motion, defined on a filtered probability
space $\left(  \Omega,\mathcal{F},\left(  \mathcal{F}_{t}\right)  _{t\geq
0},\mathbb{\mathcal{P}}\right)  ,$ satisfying the usual conditions.

The control variable $v=\left(  v_{t}\right)  $, called strict (classical)
control, is an $\mathcal{F}_{t}$ adapted process with values in some set $U$
of $\mathbb{R}^{k}$. We denote by $\mathcal{U}$ the class of all strict controls.

The criteria to be minimized, over the set $\mathcal{U}$, has the form%
\[
J\left(  v\right)  =\mathbb{E}\left[  g\left(  x_{T}^{v}\right)  +h\left(
y_{0}^{v}\right)  +%
{\displaystyle\int\nolimits_{0}^{T}}
l\left(  t,x_{t}^{v},y_{t}^{v},z_{t}^{v},v_{t}\right)  dt\right]  ,
\]
where $g,\ h$ and $l$ are given functions and $\left(  x_{t}^{v},y_{t}%
^{v},z_{t}^{v}\right)  $ is the trajectory of the system controlled by $v.$

A control $u\in\mathcal{U}$ is called optimal if it satisfies%
\[
J\left(  u\right)  =\underset{v\in\mathcal{U}}{\inf}J\left(  v\right)  .
\]

The objective of this kind of stochastic control problem is to obtain the
optimality conditions of controls in the form of Pontryagin stochastic maximum
principle. There is many works on the subject, including Peng $\left[
42\right]  $, Xu $\left[  47\right]  $, Wu $\left[  46\right]  $, Shi and Wu
$\left[  44\right]  $, Ji and Zhou $\left[  30\right]  $, Bahlali and Labed
$\left[  5\right]  $ and Bahlali $\left[  8\right]  $. All the previous
results on stochastic maximum principle of forward-backward systems are
established in the cases where the control domain is convex or uncontrolled
diffusion coefficient. The general case, where the set of strict controls need
not be convex and the diffusion coefficient depends explicitly on the control
variable, is an open problem unsolved until now. There is no result in the
literature concerning this problem and the classical way which consists to use
the spike variation method on the strict controls does not lead to any result.
The approach developed by Peng $\left[  41\right]  $ to solve the similar case
of controlled stochastic differential equations (SDEs) cannot be applied in
the case of controlled FBSDEs. Indeed, since the control domain is not
necessarily convex and the diffusion $\sigma$ depends on the control variable,
the classical way of treating such a problem would be to use the spike
variation method on the strict controls and to introduce the second-order
variational equation. But, the FBSDE system depends on three variables
($x,\ y$ and $z$) and the second order expansion leads to a nonlinear problem.
It is impossible to deduce then the second-order variational inequality.

In this paper, we solve this open problem by using the new approach developed
by Bahlali $\left[  7\right]  .$ We introduse then a bigger new class
$\mathcal{R}$ of processes by replacing the $U$-valued process $\left(
v_{t}\right)  $ by a $\mathbb{P}\left(  U\right)  $-valued process $\left(
q_{t}\right)  $, where $\mathbb{P}\left(  U\right)  $ is the space of
probability measures on $U$ equipped with the topology of stable convergence.
This new class of processes is called relaxed controls and have a richer
structure of convexity, for which the control problem becomes solvable. The
main idea is to use the property of convexity of the set of relaxed controls
and treat the problem with the method of convex perturbation on relaxed
controls (instead of that of the spike variation on strict one). We establish
then necessary and sufficient optimality conditions for relaxed controls and
we derive directly the optimality conditions for strict controls from those of
relaxed one.

\ 

In the relaxed model, the system is governed by the FBSDE%
\[
\left\{
\begin{array}
[c]{l}%
dx_{t}^{q}=\int_{U}b\left(  t,x_{t}^{q},a\right)  q_{t}\left(  da\right)
dt+\int_{U}\sigma\left(  t,x_{t}^{q},a\right)  q_{t}\left(  da\right)
dW_{t},\\
x_{0}^{q}=x,\\
dy_{t}^{q}=-\int_{U}f\left(  t,x_{t}^{q},y_{t}^{q},z_{t}^{q},a\right)
q_{t}\left(  da\right)  dt+z_{t}^{q}dW_{t},\\
y_{T}^{q}=\varphi\left(  x_{T}^{q}\right)  .
\end{array}
\right.
\]

The functional cost to be minimized, over the class $\mathcal{R}$ of relaxed
controls, is defined by%
\[
\mathcal{J}\left(  q\right)  =\mathbb{E}\left[  g\left(  x_{T}^{q}\right)
+h\left(  y_{0}^{q}\right)  +%
{\displaystyle\int\nolimits_{0}^{T}}
\int_{U}l\left(  t,x_{t}^{q},y_{t}^{q},z_{t}^{q},a\right)  q_{t}\left(
da\right)  dt\right]  .
\]

A relaxed control $\mu$ is called optimal if it solves
\[
\mathcal{J}\left(  \mu\right)  =\inf\limits_{q\in\mathcal{R}}\mathcal{J}%
\left(  q\right)  .
\]

The relaxed control problem is a generalization of the problem of strict
controls. Indeed, if $q_{t}\left(  da\right)  =\delta_{v_{t}}\left(
da\right)  $ is a Dirac measure concentrated at a single point $v_{t}\in U$,
then we get a strict control problem as a particular case of the relaxed one.

To achieve the objective of this paper and establish necessary and sufficient
optimality conditions for these two models, we proceed as follows.

Firstly, we give the optimality conditions for relaxed controls. The idea is
to use the fact that the set of relaxed controls is convex. Then, we establish
necessary optimality conditions by using the classical way of the convex
perturbation method. More precisely, if we denote by $\mu$ an optimal relaxed
control and $q$ is an arbitrary element of $\mathcal{R}$, then with a
sufficiently small $\theta>0$ and for each $t\in\left[  0,T\right]  $, we can
define a perturbed control as follows%
\[
\mu_{t}^{\theta}=\mu_{t}+\theta\left(  q_{t}-\mu_{t}\right)  .
\]

We derive the variational equation from the state equation, and the
variational inequality from the inequality%
\[
0\leq\mathcal{J}\left(  \mu^{\theta}\right)  -\mathcal{J}\left(  \mu\right)
.
\]

By using the fact that the coefficients $b,\ \sigma,\ f$ and $l$ are linear
with respect to the relaxed control variable, necessary optimality conditions
are obtained directly in the global form.

To enclose this part of the paper, we prove under minimal additional
hypothesis, that these necessary optimality conditions for relaxed controls
are also sufficient.

The second main result in the paper characterizes the optimality for strict
control processes. It is directly derived from the above result by restricting
from relaxed to strict controls. The idea is to replace the relaxed controls
by a Dirac measures charging a strict controls. Thus, we reduce the set
$\mathcal{R}$ of relaxed controls and we minimize the cost $\mathcal{J}$ over
the subset $\delta\left(  \mathcal{U}\right)  =\left\{  q\in\mathcal{R}\text{
\ / }\ q=\delta_{v}\ \ ;\ \ v\in\mathcal{U}\right\}  $. Necessary optimality
conditions for strict controls are then obtained directly from those of
relaxed one. Finally, we prove that these necessary conditions becomes
sufficient, without imposing neither the convexity of $U$ nor that of the
Hamiltonian $H$ in $v$.

This paper can be also regarded as an extension of that of Bahlali $\left[
7\right]  $ to the forward-backward systems. Indeed, if we consider only the
forward equation, without the backward one ($y=z=f=h=0$), we recover then
exactly all the results of $\left[  7\right]  .$

The paper is organized as follows. In Section 2, we formulate the strict and
relaxed control problems and give the various assumptions used throughout the
paper. Section 3 is devoted to study the relaxed control problems and we
establish necessary as well as sufficient conditions of optimality for relaxed
controls. In the last Section, we derive directly from the results of Section
3, the optimality conditions for strict controls.

\ 

Along this paper, we denote by $C$ some positive constant, $\mathcal{M}%
_{n\times d}\left(  \mathbb{R}\right)  $ the space of $n\times d$ real matrix
and $\mathcal{M}_{n\times n}^{d}\left(  \mathbb{R}\right)  $ the linear space
of vectors $M=\left(  M_{1},...,M_{d}\right)  $ where $M_{i}\in\mathcal{M}%
_{n\times n}\left(  \mathbb{R}\right)  $. We use the standard calculus of
inner and matrix product.

\section{Formulation of the problem}

Let $\left(  \Omega,\mathcal{F},\left(  \mathcal{F}_{t}\right)  _{t\geq
0},\mathcal{P}\right)  $ be a filtered probability space satisfying the usual
conditions, on which a $d$-dimensional Brownian motion $W=\left(
W_{t}\right)  _{t\geq0}$\ is defined. We assume that $\left(  \mathcal{F}%
_{t}\right)  $ is the $\mathcal{P}$- augmentation of the natural filtration of
$W.$

Let $T$ be a strictly positive real number and $U$ a non-empty set of
$\mathbb{R}^{k}$.

\subsection{The strict control problem}

\begin{definition}
\textit{An admissible strict control is an }$\mathcal{F}_{t}-$%
\textit{\ adapted process }$v=\left(  v_{t}\right)  $ \textit{with values in
}$U$\textit{\ such that }
\[
\mathbb{E}\left[  \underset{t\in\left[  0,T\right]  }{\sup}\left\vert
v_{t}\right\vert ^{2}\right]  <\infty.
\]

\textit{We denote by }$\mathcal{U}$\textit{\ the set of all admissible strict
controls.}
\end{definition}

For any $v\in\mathcal{U}$, we consider the following controlled FBSDE
\begin{equation}
\left\{
\begin{array}
[c]{l}%
dx_{t}^{v}=b\left(  t,x_{t}^{v},v_{t}\right)  dt+\sigma\left(  t,x_{t}%
^{v},v_{t}\right)  dW_{t},\\
x_{0}^{v}=x,\\
dy_{t}^{v}=-f\left(  t,x_{t}^{v},y_{t}^{v},z_{t}^{v},v_{t}\right)
dt+z_{t}^{v}dW_{t},\\
y_{T}^{v}=\varphi\left(  x_{T}^{v}\right)  ,
\end{array}
\right.
\end{equation}
where,
\begin{align*}
b  &  :\left[  0,T\right]  \times\mathbb{R}^{n}\times U\longrightarrow
\mathbb{R}^{n},\\
\sigma &  :\left[  0,T\right]  \mathbb{\times R}^{n}\times U\longrightarrow
\mathcal{M}_{n\times d}\left(  \mathbb{R}\right)  ,\\
f  &  :\left[  0,T\right]  \times\mathbb{R}^{n}\times\mathbb{R}^{m}%
\times\mathcal{M}_{m\times d}\left(  \mathbb{R}\right)  \times
U\longrightarrow\mathbb{R}^{m},\\
\varphi &  :\mathbb{R}^{n}\longrightarrow\mathbb{R}^{m},
\end{align*}
and $x$\ is an $n-$dimensional $\mathcal{F}_{0}$-measurable random variable
such that%
\[
\mathbb{E}\left\vert x\right\vert ^{2}<\infty.
\]

The criteria to be minimized is defined from $\mathcal{U}$ into $\mathbb{R}$
by%
\begin{equation}
J\left(  v\right)  =\mathbb{E}\left[  g\left(  x_{T}^{v}\right)  +h\left(
y_{0}^{v}\right)  +%
{\displaystyle\int\nolimits_{0}^{T}}
l\left(  t,x_{t}^{v},y_{t}^{v},z_{t}^{v},v_{t}\right)  dt\right]  ,
\end{equation}
where,%
\begin{align*}
g  &  :\mathbb{R}^{n}\longrightarrow\mathbb{R},\\
h  &  :\mathbb{R}^{m}\longrightarrow\mathbb{R},\\
l  &  :\left[  0,T\right]  \times\mathbb{R}^{n}\times\mathbb{R}^{m}%
\times\mathcal{M}_{m\times d}\left(  \mathbb{R}\right)  \times
U\longrightarrow\mathbb{R}.
\end{align*}

A strict control $u$ is called optimal if it satisfies%
\begin{equation}
J\left(  u\right)  =\inf\limits_{v\in\mathcal{U}}J\left(  v\right)  .
\end{equation}

\ 

We assume that%
\begin{align}
&  b,\ \sigma,\ f,\ g,\ h,\ l\text{ and }\varphi\text{ are continuously
differentiable with respect}\\
&  \text{to }\left(  x,y,z\right)  \text{, they are bounded by }C\left(
1+\left\vert x\right\vert +\left\vert y\right\vert +\left\vert z\right\vert
+\left\vert v\right\vert \right)  \text{ and their}\nonumber\\
&  \text{derivatives with respect to }\left(  x,y,z\right)  \text{ are
continuous in }\left(  x,y,z,v\right) \nonumber\\
&  \text{and uniformly bounded.}\nonumber
\end{align}

Under the above hypothesis, for every $v\in U$, equation $\left(  1\right)  $
has a unique strong solution and the functional cost $J$ is well defined from
$\mathcal{U}$ into $\mathbb{R}$.

\subsection{The relaxed model}

The idea for relaxed the strict control problem defined above is to embed the
set $U$ of strict controls into a wider class which gives a more suitable
topological structure. In the relaxed model, the $U$-valued process $v$ is
replaced by a $\mathbb{P}\left(  U\right)  $-valued process $q$, where
$\mathbb{P}\left(  U\right)  $ denotes the space of probability measure on $U$
equipped with the topology of stable convergence.

\begin{definition}
A relaxed control $\left(  q_{t}\right)  _{t}$ is a $\mathbb{P}\left(
U\right)  $-valued process, progressively measurable with respect to $\left(
\mathcal{F}_{t}\right)  _{t}$\ and such that for each $t$, $1_{]0,t]}.q$\ is
$\mathcal{F}_{t}$-measurable.

\textit{We denote by }$\mathcal{R}$\textit{\ the set of all relaxed controls.}
\end{definition}

\ 

\begin{remark}
Every relaxed control $q$ may be desintegrated as $q\left(  dt,da\right)
=q\left(  t,da\right)  dt=q_{t}\left(  da\right)  dt$, where $q_{t}\left(
da\right)  $ is a progressively measurable process with value in the set of
probability measures $\mathbb{P}(U).$

The set $U$ is embedded into the set $\mathcal{R}$\ of relaxed process by the
mapping
\[
f:v\in U\mathbb{\longmapsto}f_{v}\left(  dt,da\right)  =\delta_{v_{t}%
}(da)dt\in\mathcal{R}%
\]
where $\delta_{v}$ is the atomic measure concentrated at a single point $v$.
\end{remark}

\ 

For more details on relaxed controls, see $\left[  4\right]  ,\left[
6\right]  ,\left[  7\right]  ,\left[  16\right]  ,\left[  21\right]  ,\left[
34\right]  ,\left[  37\right]  ,\left[  38\right]  .$

\ 

For any $q\in\mathcal{R}$, we consider the following relaxed FBSDE%
\begin{equation}
\left\{
\begin{array}
[c]{l}%
dx_{t}^{q}=\int_{U}b\left(  t,x_{t}^{q},a\right)  q_{t}\left(  da\right)
dt+\int_{U}\sigma\left(  t,x_{t}^{q},a\right)  q_{t}\left(  da\right)
dW_{t},\\
x_{0}^{q}=x,\\
dy_{t}^{q}=-\int_{U}f\left(  t,x_{t}^{q},y_{t}^{q},z_{t}^{q},a\right)
q_{t}\left(  da\right)  dt+z_{t}^{q}dW_{t},\\
y_{T}^{q}=\varphi\left(  x_{T}^{q}\right)  .
\end{array}
\right.
\end{equation}

The expected cost to be minimized, in the relaxed model, is defined from
$\mathcal{R}$ into $\mathbb{R}$ by%
\begin{equation}
\mathcal{J}\left(  q\right)  =\mathbb{E}\left[  g\left(  x_{T}^{q}\right)
+h\left(  y_{0}^{q}\right)  +%
{\displaystyle\int\nolimits_{0}^{T}}
\int_{U}l\left(  t,x_{t}^{q},y_{t}^{q},z_{t}^{q},a\right)  q_{t}\left(
da\right)  dt\right]  .
\end{equation}

A relaxed control $\mu$ is called optimal if it solves%
\begin{equation}
\mathcal{J}\left(  \mu\right)  =\inf\limits_{q\in\mathcal{R}}\mathcal{J}%
\left(  q\right)  .
\end{equation}

\begin{remark}
If we put
\begin{align*}
\overline{b}\left(  t,x_{t}^{q},q_{t}\right)   &  =\int_{U}b\left(
t,x_{t}^{q},a\right)  q_{t}\left(  da\right)  ,\\
\overline{\sigma}\left(  t,x_{t}^{q},q_{t}\right)   &  =\int_{U}\sigma\left(
t,x_{t}^{q},a\right)  q_{t}\left(  da\right)  ,\\
\overline{f}\left(  t,x_{t}^{q},y_{t}^{q},z_{t}^{q},a\right)   &  =\int
_{U}f\left(  t,x_{t}^{q},y_{t}^{q},z_{t}^{q},a\right)  q_{t}\left(  da\right)
,\\
\overline{l}\left(  t,x_{t}^{q},y_{t}^{q},z_{t}^{q},a\right)   &  =\int
_{U}l\left(  t,x_{t}^{q},y_{t}^{q},z_{t}^{q},a\right)  q_{t}\left(  da\right)
.
\end{align*}

Then, equation $\left(  5\right)  $ becomes
\[
\left\{
\begin{array}
[c]{l}%
dx_{t}^{q}=\overline{b}\left(  t,x_{t}^{q},q_{t}\right)  dt+\overline{\sigma
}\left(  t,x_{t}^{q},q_{t}\right)  dW_{t},\\
x_{0}^{q}=x,\\
dy_{t}^{q}=-\overline{f}\left(  t,x_{t}^{q},y_{t}^{q},z_{t}^{q},q_{t}\right)
dt+z_{t}^{q}dW_{t},\\
y_{T}^{q}=\varphi\left(  x_{T}^{q}\right)  .
\end{array}
\right.
\]

With a functional cost given by%
\[
\mathcal{J}\left(  q\right)  =\mathbb{E}\left[  g\left(  x_{T}^{q}\right)
+h\left(  y_{0}^{q}\right)  +%
{\displaystyle\int\nolimits_{0}^{T}}
\overline{l}\left(  t,x_{t}^{q},y_{t}^{q},z_{t}^{q},q_{t}\right)  dt\right]
.
\]

Hence, by introducing relaxed controls, we have replaced $U$ by a larger space
$\mathbb{P}\left(  U\right)  $. We have gained the advantage that
$\mathbb{P}\left(  U\right)  $ is and convex. Furthermore, the new
coefficients of equation $\left(  5\right)  $ and the running cost are linear
with respect to the relaxed control variable.
\end{remark}

\begin{remark}
The coefficients $\overline{b},\overline{\sigma}$ and $\overline{f}$ (defined
in the above remark) check respectively the same assumptions as $b,\sigma$ and
$f$. Then, under assumptions $\left(  4\right)  $, $\overline{b}%
,\overline{\sigma}$ and $\overline{f}$ are uniformly Lipschitz and with linear
growth. Then by classical results on FBSDEs, for every $q\in\mathcal{R}$
equation $\left(  5\right)  $ has a unique strong solution.

On the other hand, It is easy to see that $\overline{l}$ checks the same
assumptions as $l$. Then, the functional cost $\mathcal{J}$ is well defined
from $\mathcal{R}$ into $\mathbb{R}$.
\end{remark}

\begin{remark}
If $q_{t}=\delta_{v_{t}}$ is an atomic measure concentrated at a single point
$v_{t}\in U$, then for each $t\in\left[  0,T\right]  $ we have
\begin{align*}
\int_{U}b\left(  t,x_{t}^{q},a\right)  q_{t}\left(  da\right)   &  =\int
_{U}b\left(  t,x_{t}^{q},a\right)  \delta_{v_{t}}\left(  da\right)  =b\left(
t,x_{t}^{q},v_{t}\right)  ,\\
\int_{U}\sigma\left(  t,x_{t}^{q},a\right)  q_{t}\left(  da\right)   &
=\int_{U}\sigma\left(  t,x_{t}^{q},a\right)  \delta_{v_{t}}\left(  da\right)
=\sigma\left(  t,x_{t}^{q},v_{t}\right)  ,\\
\int_{U}f\left(  t,x_{t}^{q},y_{t}^{q},z_{t}^{q},a\right)  q_{t}\left(
da\right)   &  =\int_{U}f\left(  t,x_{t}^{q},y_{t}^{q},z_{t}^{q},a\right)
\delta_{v_{t}}\left(  da\right)  =f\left(  t,x_{t}^{q},y_{t}^{q},z_{t}%
^{q},v_{t}\right)  ,\\
\int_{U}l\left(  t,x_{t}^{q},y_{t}^{q},z_{t}^{q},a\right)  q_{t}\left(
da\right)   &  =\int_{U}l\left(  t,x_{t}^{q},y_{t}^{q},z_{t}^{q},a\right)
\delta_{v_{t}}\left(  da\right)  =l\left(  t,x_{t}^{q},y_{t}^{q},z_{t}%
^{q},v_{t}\right)  .
\end{align*}

In this case $\left(  x^{q},y^{q},z^{q}\right)  =\left(  x^{v},y^{v}%
,z^{v}\right)  $, $\mathcal{J}\left(  q\right)  =J\left(  v\right)  $ and we
get a strict control problem. So the problem of strict controls $\left\{
\left(  1\right)  ,\left(  2\right)  ,\left(  3\right)  \right\}  $ is a
particular case of relaxed control problem $\left\{  \left(  5\right)
,\left(  6\right)  ,\left(  7\right)  \right\}  $.
\end{remark}

\begin{remark}
The relaxed control problems studied in El Karoui et al $\left[  16\right]  $
and Bahlali, Mezerdi and Djehiche $\left[  4\right]  $ is different to ours,
in that they relax the corresponding infinitesimal generator of the state
process, which leads to a martingale problem for which the state process
driven by an orthogonal martingale measure. In our setting the driving
martingale measure $q_{t}\left(  da\right)  dW_{t}$ is however not orthogonal.
See Ma and Yong $\left[  34\right]  $ for more details.
\end{remark}

\section{Optimality conditions for relaxed controls}

In this section, we study the problem $\left\{  \left(  5\right)  ,\left(
6\right)  ,\left(  7\right)  \right\}  $ and we establish necessary as well as
sufficient conditions of optimality for relaxed controls.

\subsection{Preliminary results}

Since the set $\mathcal{R}$ is convex, then the classical way to derive
necessary optimality conditions for relaxed controls is to use the convex
perturbation method. More precisely, let $\mu$ be an optimal relaxed control
and $\left(  x^{\mu},y^{\mu},z^{\mu}\right)  $ the solution of $\left(
5\right)  $ controlled by $\mu$. Then, for each $t\in\left[  0,T\right]  $ we
can define a perturbed relaxed control as follows%
\[
\mu_{t}^{\theta}=\mu_{t}+\theta\left(  q_{t}-\mu_{t}\right)  ,
\]
where, $\theta>0$ is sufficiently small and $q$ is an arbitrary element of
$\mathcal{R}$.

Denote by $\left(  x^{\theta},y^{\theta},z^{\theta}\right)  $ the solution of
$\left(  5\right)  $ associated with $\mu^{\theta}$.

From optimality of $\mu$, the variational inequality will be derived from the
fact that
\[
0\leq\mathcal{J}\left(  \mu^{\theta}\right)  -\mathcal{J}\left(  \mu\right)
.
\]

For this end, we need the following classical lemmas.

\begin{lemma}
\textit{Under assumptions }$\left(  4\right)  $, we have%
\begin{align}
\underset{\theta\rightarrow0}{\lim}\left[  \underset{t\in\left[  0,T\right]
}{\sup}\mathbb{E}\left\vert x_{t}^{\theta}-x_{t}^{\mu}\right\vert ^{2}\right]
&  =0,\\
\underset{\theta\rightarrow0}{\lim}\left[  \underset{t\in\left[  0,T\right]
}{\sup}\mathbb{E}\left\vert y_{t}^{\theta}-y_{t}^{\mu}\right\vert ^{2}\right]
&  =0,\\
\underset{\theta\rightarrow0}{\lim}\mathbb{E}%
{\displaystyle\int\nolimits_{0}^{T}}
\left\vert z_{t}^{\theta}-z_{t}^{\mu}\right\vert ^{2}dt  &  =0.
\end{align}

\end{lemma}

\begin{proof}
$\left(  8\right)  $ is proved in $\left[  7\text{, Lemm 9, page 2085}\right]
.$

\ 

Let us prove $\left(  9\right)  $\ and $\left(  10\right)  $.

Applying It\^{o}'s formula to $\left(  y_{t}^{\theta}-y_{t}^{\mu}\right)
^{2}$, we have%
\[%
\begin{array}
[c]{c}%
\mathbb{E}\left\vert y_{t}^{\theta}-y_{t}^{\mu}\right\vert ^{2}+\mathbb{E}%
{\displaystyle\int\nolimits_{t}^{T}}
\left\vert z_{s}^{\theta}-z_{s}^{\mu}\right\vert ^{2}ds=\mathbb{E}\left\vert
\varphi\left(  x_{T}^{\theta}\right)  -\varphi\left(  x_{T}^{\mu}\right)
\right\vert ^{2}\\
+2\mathbb{E}%
{\displaystyle\int\nolimits_{t}^{T}}
\left\vert \left(  y_{s}^{\theta}-y_{s}^{\mu}\right)  \left[
{\displaystyle\int\nolimits_{U}}
f\left(  s,x_{s}^{\theta},y_{s}^{\theta},z_{s}^{\theta},a\right)  \mu
_{s}^{\theta}\left(  da\right)  -%
{\displaystyle\int\nolimits_{U}}
f\left(  s,x_{s}^{\mu},y_{s}^{\mu},z_{s}^{\mu},a\right)  \mu_{s}\left(
da\right)  \right]  \right\vert \,ds.
\end{array}
\]

From the Young formula, for every $\varepsilon>0$, we have%
\begin{align*}
&  \mathbb{E}\left\vert y_{t}^{\theta}-y_{t}^{\mu}\right\vert ^{2}+\mathbb{E}%
{\displaystyle\int\nolimits_{t}^{T}}
\left\vert z_{s}^{\theta}-z_{s}^{\mu}\right\vert ^{2}ds\\
&  \leq\mathbb{E}\left\vert \varphi\left(  x_{T}^{\theta}\right)
-\varphi\left(  x_{T}^{\mu}\right)  \right\vert ^{2}+%
\genfrac{.}{.}{}{0}{1}{\varepsilon}%
\mathbb{E}%
{\displaystyle\int\nolimits_{t}^{T}}
\left\vert y_{s}^{\theta}-y_{s}^{\mu}\right\vert ^{2}ds\\
&  +\varepsilon\mathbb{E}%
{\displaystyle\int\nolimits_{t}^{T}}
\left\vert
{\displaystyle\int\nolimits_{U}}
f\left(  s,x_{s}^{\theta},y_{s}^{\theta},z_{s}^{\theta},a\right)  \mu
_{s}^{\theta}\left(  da\right)  -%
{\displaystyle\int\nolimits_{U}}
f\left(  s,x_{s}^{\mu},y_{s}^{\mu},z_{s}^{\mu},a\right)  \mu_{s}\left(
da\right)  \right\vert ^{2}ds.
\end{align*}

Then,%
\begin{align*}
&  \mathbb{E}\left\vert y_{t}^{\theta}-y_{t}^{\mu}\right\vert ^{2}+\mathbb{E}%
{\displaystyle\int\nolimits_{t}^{T}}
\left\vert z_{s}^{\theta}-z_{s}^{\mu}\right\vert ^{2}ds\\
&  \leq\mathbb{E}\left\vert \varphi\left(  x_{T}^{\theta}\right)
-\varphi\left(  x_{T}^{\mu}\right)  \right\vert ^{2}+%
\genfrac{.}{.}{}{0}{1}{\varepsilon}%
\mathbb{E}%
{\displaystyle\int\nolimits_{t}^{T}}
\left\vert y_{s}^{\theta}-y_{s}^{\mu}\right\vert ^{2}ds\\
&  +C\varepsilon\mathbb{E}%
{\displaystyle\int\nolimits_{t}^{T}}
\left\vert
{\displaystyle\int\nolimits_{U}}
f\left(  s,x_{s}^{\theta},y_{s}^{\theta},z_{s}^{\theta},a\right)  \mu
_{s}^{\theta}\left(  da\right)  -%
{\displaystyle\int\nolimits_{U}}
f\left(  s,x_{s}^{\theta},y_{s}^{\theta},z_{s}^{\theta},a\right)  \mu
_{s}\left(  da\right)  \right\vert ^{2}ds\\
&  +C\varepsilon\mathbb{E}%
{\displaystyle\int\nolimits_{t}^{T}}
\left\vert
{\displaystyle\int\nolimits_{U}}
f\left(  s,x_{s}^{\theta},y_{s}^{\theta},z_{s}^{\theta},a\right)  \mu
_{s}\left(  da\right)  -%
{\displaystyle\int\nolimits_{U}}
f\left(  s,x_{s}^{\mu},y_{s}^{\theta},z_{s}^{\theta},a\right)  \mu_{s}\left(
da\right)  \right\vert ^{2}ds\\
&  +C\varepsilon\mathbb{E}%
{\displaystyle\int\nolimits_{t}^{T}}
\left\vert
{\displaystyle\int\nolimits_{U}}
f\left(  s,x_{s}^{\mu},y_{s}^{\theta},z_{s}^{\theta},a\right)  \mu_{s}\left(
da\right)  -%
{\displaystyle\int\nolimits_{U}}
f\left(  s,x_{s}^{\mu},y_{s}^{\mu},z_{s}^{\theta},a\right)  \mu_{s}\left(
da\right)  \right\vert ^{2}ds\\
&  +C\varepsilon\mathbb{E}%
{\displaystyle\int\nolimits_{t}^{T}}
\left\vert
{\displaystyle\int\nolimits_{U}}
f\left(  s,x_{s}^{\mu},y_{s}^{\mu},z_{s}^{\theta},a\right)  \mu_{s}\left(
da\right)  -%
{\displaystyle\int\nolimits_{U}}
f\left(  s,x_{s}^{\mu},y_{s}^{\mu},z_{s}^{\mu},a\right)  \mu_{s}\left(
da\right)  \right\vert ds.
\end{align*}

By the definition of $\mu_{t}^{\theta}$, we have%
\begin{align*}
&  \mathbb{E}\left\vert y_{t}^{\theta}-y_{t}^{\mu}\right\vert ^{2}+\mathbb{E}%
{\displaystyle\int\nolimits_{t}^{T}}
\left\vert z_{s}^{\theta}-z_{s}^{\mu}\right\vert ^{2}ds\\
&  \leq\mathbb{E}\left\vert \varphi\left(  x_{T}^{\theta}\right)
-\varphi\left(  x_{T}\right)  \right\vert ^{2}+%
\genfrac{.}{.}{}{0}{1}{\varepsilon}%
\mathbb{E}%
{\displaystyle\int\nolimits_{t}^{T}}
\left\vert y_{s}^{\theta}-y_{s}^{\mu}\right\vert ^{2}ds\\
&  +C\varepsilon\theta^{2}\mathbb{E}%
{\displaystyle\int\nolimits_{t}^{T}}
\left\vert
{\displaystyle\int\nolimits_{U}}
f\left(  s,x_{s}^{\theta},y_{s}^{\theta},z_{s}^{\theta},a\right)  q_{s}\left(
da\right)  -%
{\displaystyle\int\nolimits_{U}}
f\left(  s,x_{s}^{\theta},y_{s}^{\theta},z_{s}^{\theta},a\right)  \mu
_{s}\left(  da\right)  \right\vert ^{2}ds\\
&  +C\varepsilon\mathbb{E}%
{\displaystyle\int\nolimits_{t}^{T}}
\left\vert
{\displaystyle\int\nolimits_{U}}
f\left(  s,x_{s}^{\theta},y_{s}^{\theta},z_{s}^{\theta},a\right)  \mu
_{s}\left(  da\right)  -%
{\displaystyle\int\nolimits_{U}}
f\left(  s,x_{s}^{\mu},y_{s}^{\theta},z_{s}^{\theta},a\right)  \mu_{s}\left(
da\right)  \right\vert ^{2}ds\\
&  +C\varepsilon\mathbb{E}%
{\displaystyle\int\nolimits_{t}^{T}}
\left\vert
{\displaystyle\int\nolimits_{U}}
f\left(  s,x_{s}^{\mu},y_{s}^{\theta},z_{s}^{\theta},a\right)  \mu_{s}\left(
da\right)  -%
{\displaystyle\int\nolimits_{U}}
f\left(  s,x_{s}^{\mu},y_{s}^{\mu},z_{s}^{\theta},a\right)  \mu_{s}\left(
da\right)  \right\vert ^{2}ds\\
&  +C\varepsilon\mathbb{E}%
{\displaystyle\int\nolimits_{t}^{T}}
\left\vert
{\displaystyle\int\nolimits_{U}}
f\left(  s,x_{s}^{\mu},y_{s}^{\mu},z_{s}^{\theta},a\right)  \mu_{s}\left(
da\right)  -%
{\displaystyle\int\nolimits_{U}}
f\left(  s,x_{s}^{\mu},y_{s}^{\mu},z_{s}^{\mu},a\right)  \mu_{s}\left(
da\right)  \right\vert ^{2}ds.
\end{align*}

Since $\varphi$\ and $f$\ are uniformly Lipschitz with respect to $x,y,z$,
then%
\begin{align}
\mathbb{E}\left\vert y_{t}^{\theta}-y_{t}^{\mu}\right\vert ^{2}+\mathbb{E}%
{\displaystyle\int\nolimits_{t}^{T}}
\left\vert z_{s}^{\theta}-z_{s}^{\mu}\right\vert ^{2}ds  &  \leq\left(
\genfrac{.}{.}{}{0}{1}{\varepsilon}%
\mathbb{+}C\varepsilon\right)  \mathbb{E}%
{\displaystyle\int\nolimits_{t}^{T}}
\left\vert y_{s}^{\theta}-y_{s}^{\mu}\right\vert ^{2}ds\\
&  +C\varepsilon\mathbb{E}%
{\displaystyle\int\nolimits_{t}^{T}}
\left\vert z_{s}^{\theta}-z_{s}^{\mu}\right\vert ^{2}ds+\alpha_{t}^{\theta
},\nonumber
\end{align}
where $\alpha_{t}^{\theta}$\ is given by%
\[
\alpha_{t}^{\theta}=\mathbb{E}\left\vert x_{T}^{\theta}-x_{T}^{\mu}\right\vert
^{2}+C\varepsilon\mathbb{E}%
{\displaystyle\int\nolimits_{t}^{T}}
\left\vert x_{s}^{\theta}-x_{s}^{\mu}\right\vert ^{2}ds+C\varepsilon\theta
^{2}.
\]

By $\left(  8\right)  $, we have%
\begin{equation}
\underset{\theta\rightarrow0}{\lim}\alpha_{t}^{\theta}=0.
\end{equation}

Choose $\varepsilon=%
\genfrac{.}{.}{}{0}{1}{2C}%
$, then $\left(  11\right)  $\ becomes%
\[
\mathbb{E}\left\vert y_{t}^{\theta}-y_{t}^{\mu}\right\vert ^{2}+%
\genfrac{.}{.}{}{0}{1}{2}%
\mathbb{E}%
{\displaystyle\int\nolimits_{t}^{T}}
\left\vert z_{s}^{\theta}-z_{s}^{\mu}\right\vert ^{2}ds\leq\left(  2C+%
\genfrac{.}{.}{}{0}{1}{2}%
\right)  \mathbb{E}%
{\displaystyle\int\nolimits_{t}^{T}}
\left\vert y_{s}^{\theta}-y_{s}^{\mu}\right\vert ^{2}ds+\alpha_{t}^{\theta}.
\]

From the above inequality, we derive two inequalities%
\begin{equation}
\mathbb{E}\left\vert y_{t}^{\theta}-y_{t}^{\mu}\right\vert ^{2}\leq\left(  2C+%
\genfrac{.}{.}{}{0}{1}{2}%
\right)  \mathbb{E}%
{\displaystyle\int\nolimits_{t}^{T}}
\left\vert y_{s}^{\theta}-y_{s}^{\mu}\right\vert ^{2}ds+\alpha_{t}^{\theta},
\end{equation}%
\begin{equation}
\mathbb{E}%
{\displaystyle\int\nolimits_{t}^{T}}
\left\vert z_{s}^{\theta}-z_{s}^{\mu}\right\vert ^{2}ds\leq\left(
4C+1\right)  \mathbb{E}%
{\displaystyle\int\nolimits_{t}^{T}}
\left\vert y_{s}^{\theta}-y_{s}^{\mu}\right\vert ^{2}ds+2\alpha_{t}^{\theta}.
\end{equation}

By using $\left(  12\right)  ,\ \left(  13\right)  $, Gronwall's lemma and
Bukholder-Davis-Gundy inequality, we obtain $\left(  9\right)  $.\ Finally,
$\left(  10\right)  $\ is derived from $\left(  9\right)  $\ and $\left(
12\right)  $.
\end{proof}

\begin{lemma}
\textit{Let }$\widetilde{x}_{t}$\textit{\ and }$\widetilde{y}_{t}$ are
respectively \textit{the solutions of the following linear equations (called
variational equations)}%
\begin{equation}
\left\{
\begin{array}
[c]{ll}%
d\widetilde{x}_{t}= &
{\displaystyle\int\nolimits_{U}}
b_{x}\left(  t,x_{t}^{\mu},a\right)  \mu_{t}\left(  da\right)  \widetilde
{x}_{t}dt+%
{\displaystyle\int\nolimits_{U}}
\sigma_{x}\left(  t,x_{t}^{\mu},a\right)  \mu_{t}\left(  da\right)
\widetilde{x}_{t}dW_{t}\\
& +\left[
{\displaystyle\int\nolimits_{U}}
b\left(  t,x_{t}^{\mu},a\right)  \mu_{t}\left(  da\right)  -%
{\displaystyle\int\nolimits_{U}}
b\left(  t,x_{t}^{\mu},a\right)  q_{t}\left(  da\right)  \right]  dt\\
& +\left[
{\displaystyle\int\nolimits_{U}}
\sigma\left(  t,x_{t}^{\mu},a\right)  \mu_{t}\left(  da\right)  -%
{\displaystyle\int\nolimits_{U}}
\sigma\left(  t,x_{t}^{\mu},a\right)  q_{t}\left(  da\right)  \right]
dW_{t},\\
\widetilde{x}_{0}= & 0.
\end{array}
\right.
\end{equation}
\ \
\begin{equation}
\left\{
\begin{array}
[c]{ll}%
d\widetilde{y}_{t}= & -%
{\displaystyle\int\nolimits_{U}}
\left[  f_{x}\left(  t,x_{t}^{\mu},y_{t}^{\mu},z_{t}^{\mu},a\right)
\widetilde{x}_{t}+f_{y}\left(  t,x_{t}^{\mu},y_{t}^{\mu},z_{t}^{\mu},a\right)
\widetilde{y}_{t}+f_{z}\left(  t,x_{t}^{\mu},y_{t}^{\mu},z_{t}^{\mu},a\right)
\widetilde{z}_{t}\right]  \mu_{t}\left(  da\right)  dt\\
& +\left[
{\displaystyle\int\nolimits_{U}}
f\left(  t,x_{t}^{\mu},y_{t}^{\mu},z_{t}^{\mu},a\right)  \mu_{t}\left(
da\right)  -%
{\displaystyle\int\nolimits_{U}}
f\left(  t,x_{t}^{\mu},y_{t}^{\mu},z_{t}^{\mu},a\right)  q_{t}\left(
da\right)  \right]  dt+\widetilde{z}_{t}dW_{t},\\
\widetilde{y}_{T}= & \varphi_{x}\left(  x_{T}^{\mu}\right)  \widetilde{x}_{T}.
\end{array}
\right.
\end{equation}

T\textit{hen, the following estimations hold}%
\begin{align}
\underset{\theta\rightarrow0}{\lim}\mathbb{E}\left\vert
\genfrac{.}{.}{}{0}{x_{t}^{\theta}-x_{t}^{\mu}}{\theta}%
-\widetilde{x}_{t}\right\vert ^{2}  &  =0,\\
\underset{\theta\rightarrow0}{\lim}\mathbb{E}\left\vert
\genfrac{.}{.}{}{0}{y_{t}^{\theta}-y_{t}^{\mu}}{\theta}%
-\widetilde{y}_{t}\right\vert ^{2}  &  =0,\\
\underset{\theta\rightarrow0}{\lim}\mathbb{E}%
{\displaystyle\int\nolimits_{0}^{T}}
\left\vert
\genfrac{.}{.}{}{0}{z_{t}^{\theta}-z_{t}^{\mu}}{\theta}%
-\widetilde{z}_{t}\right\vert ^{2}dt  &  =0.
\end{align}

\end{lemma}

\begin{proof}
For simplicity, we put%
\begin{align}
X_{t}^{\theta}  &  =%
\genfrac{.}{.}{}{0}{x_{t}^{\theta}-x_{t}^{\mu}}{\theta}%
-\widetilde{x}_{t},\\
Y_{t}^{\theta}  &  =%
\genfrac{.}{.}{}{0}{y_{t}^{\theta}-y_{t}^{\mu}}{\theta}%
-\widetilde{y}_{t},\\
Z_{t}^{\theta}  &  =%
\genfrac{.}{.}{}{0}{z_{t}^{\theta}-z_{t}^{\mu}}{\theta}%
-\widetilde{z}_{t}.
\end{align}%
\[
\Lambda_{t}^{\theta}\left(  a\right)  =\left(  t,x_{t}^{\mu}+\lambda
\theta\left(  X_{t}^{\theta}+\widetilde{x}_{t}\right)  ,y_{t}^{\mu}%
+\lambda\theta\left(  Y_{t}^{\theta}+\widetilde{y}_{t}\right)  ,z_{t}^{\mu
}+\lambda\theta\left(  Z_{t}^{\theta}+\widetilde{z}_{t}\right)  ,a\right)  .
\]

i) $\left(  17\right)  $ is proved in $\left[  7\text{, Lemma 10, Page
2086}\right]  $

\ 

ii) Proof of $\left(  18\right)  $ and $\left(  19\right)  $.

By $\left(  21\right)  $ and $\left(  22\right)  $, we have the following
FBSDE%
\[
\left\{
\begin{array}
[c]{l}%
dY_{t}^{\theta}=\left(  F_{t}^{y}Y_{t}^{\theta}dt+F_{t}^{y}Z_{t}^{\theta
}-\gamma_{t}^{\theta}\right)  dt+Z_{t}^{\theta}dW_{t},\\
Y_{T}^{\theta}=\dfrac{\varphi\left(  x_{T}^{\theta}\right)  -\varphi\left(
x_{T}^{\mu}\right)  }{\theta}-\varphi_{x}\left(  x_{T}^{\mu}\right)
\widetilde{x}_{T},
\end{array}
\right.
\]
where,%
\begin{align*}
F_{t}^{y}  &  =-\int_{0}^{1}\int_{U}f_{y}\left(  \Lambda_{t}^{\theta}\left(
a\right)  \right)  \mu_{t}\left(  da\right)  d\lambda,\\
F_{t}^{z}  &  =-\int_{0}^{1}\int_{U}f_{z}\left(  \Lambda_{t}^{\theta}\left(
a\right)  \right)  \mu_{t}\left(  da\right)  d\lambda,
\end{align*}
and $\gamma_{t}^{\theta}$ is given by%
\begin{align*}
&  \gamma_{t}^{\theta}=\int_{t}^{T}\int_{U}f_{x}\left(  \Lambda_{s}^{\theta
}\left(  a\right)  \right)  X_{s}^{\theta}\mu_{s}\left(  da\right)  ds\\
&  +\int_{t}^{T}\int_{U}\left[  f_{x}\left(  \Lambda_{s}^{\theta}\left(
a\right)  \right)  \left(  x_{s}^{\theta}-x_{s}^{\mu}\right)  +f_{y}\left(
\Lambda_{s}^{\theta}\left(  a\right)  \right)  \left(  y_{s}^{\theta}%
-y_{s}^{\mu}\right)  +f_{z}\left(  \Lambda_{s}^{\theta}\left(  a\right)
\right)  \left(  z_{s}^{\theta}-z_{s}^{\mu}\right)  \right]  q_{s}\left(
da\right)  ds\\
&  -\int_{t}^{T}\int_{U}\left[  f_{x}\left(  \Lambda_{s}^{\theta}\left(
a\right)  \right)  \left(  x_{s}^{\theta}-x_{s}^{\mu}\right)  +f_{y}\left(
\Lambda_{s}^{\theta}\left(  a\right)  \right)  \left(  y_{s}^{\theta}%
-y_{s}^{\mu}\right)  +f_{z}\left(  \Lambda_{s}^{\theta}\left(  a\right)
\right)  \left(  z_{s}^{\theta}-z_{s}^{\mu}\right)  \right]  \mu_{s}\left(
da\right)  ds.
\end{align*}

Since $f_{x},$ $f_{y}$ and $f_{z}$ are continuous and bounded, then from
$\left(  8\right)  ,\ \left(  9\right)  ,\ \left(  10\right)  $ and $\left(
17\right)  $, we have%
\begin{equation}
\underset{\theta\rightarrow0}{\lim}\mathbb{E}\left\vert \gamma_{t}^{\theta
}\right\vert ^{2}=0.
\end{equation}

Applying It\^{o}'s formula to $\left(  Y_{t}^{\theta}\right)  ^{2}$, we get%
\[
\mathbb{E}\left\vert Y_{t}^{\theta}\right\vert ^{2}+\mathbb{E}%
{\displaystyle\int\nolimits_{t}^{T}}
\left\vert Z_{s}^{\theta}\right\vert ^{2}ds=\mathbb{E}\left\vert Y_{T}%
^{\theta}\right\vert ^{2}+2\mathbb{E}%
{\displaystyle\int\nolimits_{t}^{T}}
\left\vert Y_{s}^{\theta}\left(  F_{s}^{y}Y_{s}^{\theta}+F_{s}^{z}%
Z_{s}^{\theta}-\gamma_{s}^{\theta}\right)  \right\vert ds.
\]

By using the Young formula, for every $\varepsilon>0$, we have%
\begin{align*}
\mathbb{E}\left\vert Y_{t}^{\theta}\right\vert ^{2}+\mathbb{E}%
{\displaystyle\int\nolimits_{t}^{T}}
\left\vert Z_{s}^{\theta}\right\vert ^{2}ds  &  \leq\mathbb{E}\left\vert
Y_{T}^{\theta}\right\vert ^{2}+%
\genfrac{.}{.}{}{0}{1}{\varepsilon}%
\mathbb{E}%
{\displaystyle\int\nolimits_{t}^{T}}
\left\vert Y_{s}^{\theta}\right\vert ^{2}ds+\varepsilon\mathbb{E}%
{\displaystyle\int\nolimits_{t}^{T}}
\left\vert \left(  F_{s}^{y}Y_{s}^{\theta}+F_{s}^{z}Z_{s}^{\theta}-\gamma
_{s}^{\theta}\right)  \right\vert ^{2}ds\\
&  \leq\mathbb{E}\left\vert Y_{T}^{\theta}\right\vert ^{2}+%
\genfrac{.}{.}{}{0}{1}{\varepsilon}%
\mathbb{E}%
{\displaystyle\int\nolimits_{t}^{T}}
\left\vert Y_{s}^{\theta}\right\vert ^{2}ds+C\varepsilon\mathbb{E}%
{\displaystyle\int\nolimits_{t}^{T}}
\left\vert F_{s}^{y}Y_{s}^{\theta}\right\vert ^{2}ds\\
&  +C\varepsilon\mathbb{E}%
{\displaystyle\int\nolimits_{t}^{T}}
\left\vert F_{s}^{z}Z_{s}^{\theta}\right\vert ^{2}ds+C\varepsilon\mathbb{E}%
{\displaystyle\int\nolimits_{t}^{T}}
\left\vert \gamma_{s}^{\theta}\right\vert ^{2}ds.
\end{align*}

Since $F_{t}^{y}$ and $F_{t}^{z}$ are bounded, then%
\[
\mathbb{E}\left\vert Y_{t}^{\theta}\right\vert ^{2}+\mathbb{E}%
{\displaystyle\int\nolimits_{t}^{T}}
\left\vert Z_{s}^{\theta}\right\vert ^{2}ds\leq\left(
\genfrac{.}{.}{}{0}{1}{\varepsilon}%
+C\varepsilon\right)  \mathbb{E}%
{\displaystyle\int\nolimits_{t}^{T}}
\left\vert Y_{s}^{\theta}\right\vert ^{2}ds+C\varepsilon\mathbb{E}%
{\displaystyle\int\nolimits_{t}^{T}}
\left\vert Z_{s}^{\theta}\right\vert ^{2}ds+\eta_{t}^{\theta},
\]
where%
\[
\eta_{t}^{\theta}=\mathbb{E}\left\vert Y_{T}^{\theta}\right\vert
^{2}+C\varepsilon\mathbb{E}%
{\displaystyle\int\nolimits_{t}^{T}}
\left\vert \gamma_{s}^{\theta}\right\vert ^{2}ds.
\]

Choose $\varepsilon=%
\genfrac{.}{.}{}{0}{1}{2C}%
$, then we have%
\[
\mathbb{E}\left\vert Y_{t}^{\theta}\right\vert ^{2}+%
\genfrac{.}{.}{}{0}{1}{2}%
\mathbb{E}%
{\displaystyle\int\nolimits_{t}^{T}}
\left\vert Z_{s}^{\theta}\right\vert ^{2}ds\leq\left(  2C+%
\genfrac{.}{.}{}{0}{1}{2}%
\right)  \mathbb{E}%
{\displaystyle\int\nolimits_{t}^{T}}
\left\vert Y_{s}^{\theta}\right\vert ^{2}ds+\eta_{t}^{\theta}.
\]

From the above inequality, we deduce two inequalities%
\begin{equation}
\mathbb{E}\left\vert Y_{t}^{\theta}\right\vert ^{2}\leq\left(  2C+%
\genfrac{.}{.}{}{0}{1}{2}%
\right)  \mathbb{E}%
{\displaystyle\int\nolimits_{t}^{T}}
\left\vert Y_{s}^{\theta}\right\vert ^{2}ds+\eta_{t}^{\theta},
\end{equation}%
\begin{equation}
\mathbb{E}%
{\displaystyle\int\nolimits_{t}^{T}}
\left\vert Z_{s}^{\theta}\right\vert ^{2}ds\leq\left(  4C+1\right)  \mathbb{E}%
{\displaystyle\int\nolimits_{t}^{T}}
\left\vert Y_{s}^{\theta}\right\vert ^{2}ds+2\eta_{t}^{\theta}.
\end{equation}

On the other hand, we have%
\begin{align*}
\mathbb{E}\left\vert Y_{T}^{\theta}\right\vert ^{2}  &  =\mathbb{E}\left\vert
\widetilde{y}_{T}-%
\genfrac{.}{.}{}{0}{y_{T}^{\theta}-y_{T}^{\mu}}{\theta}%
\right\vert ^{2}\\
&  =\mathbb{E}\left\vert \varphi_{x}\left(  x_{T}^{\mu}\right)  \widetilde
{x}_{T}-%
\genfrac{.}{.}{}{0}{\varphi\left(  x_{T}^{\theta}\right)  -\varphi\left(
x_{T}^{\mu}\right)  }{\theta}%
\right\vert ^{2}\\
&  \leq2\mathbb{E}%
{\displaystyle\int\nolimits_{0}^{1}}
\left\vert \left[  \varphi_{x}\left(  x_{T}^{\mu}\right)  -\varphi_{x}\left(
x_{T}^{\mu}+\lambda\theta\left(  \widetilde{x}_{T}+X_{T}^{\theta}\right)
\right)  \right]  \widetilde{x}_{T}\right\vert ^{2}d\lambda\\
&  +2\mathbb{E}%
{\displaystyle\int\nolimits_{0}^{1}}
\left\vert \varphi_{x}\left(  x_{T}^{\mu}+\lambda\theta\left(  \widetilde
{x}_{T}+X_{T}^{\theta}\right)  \right)  X_{T}^{\theta}\right\vert ^{2}%
d\lambda.
\end{align*}

By using $\left(  17\right)  $ and the fact that $\varphi_{x}$ is continuous
and bounded, we get%
\begin{equation}
\underset{\theta\rightarrow0}{\lim}\mathbb{E}\left\vert Y_{T}^{\theta
}\right\vert ^{2}=0.
\end{equation}

From $\left(  23\right)  $ and $\left(  26\right)  $, we deduce that%
\begin{equation}
\underset{\theta\rightarrow0}{\lim}\eta_{t}^{\theta}=0.
\end{equation}

Finally, by using $\left(  24\right)  ,\ \left(  27\right)  $, Gronwall's
lemma and Bukholder-Davis-Gundy inequality, we obtain $\left(  18\right)  $.
Finally $\left(  19\right)  $ is derived from $\left(  25\right)  ,\ \left(
27\right)  $ and $\left(  18\right)  $.
\end{proof}

\begin{lemma}
\textit{Let }$\mu$\textit{\ be an optimal control minimizing the functional
}$\mathcal{J}$\textit{\ over }$\mathcal{R}$\textit{\ and }$\left(  x_{t}^{\mu
},y_{t}^{\mu},z_{t}^{\mu}\right)  $\textit{\ the solution of }$\left(
1\right)  $\textit{\ associated with }$\mu$\textit{. Then for any }%
$q\in\mathcal{R}$\textit{, we have}
\begin{align}
0  &  \leq\mathbb{E}\left[  g_{x}\left(  x_{T}^{\mu}\right)  \widetilde{x}%
_{T}\right]  +\mathbb{E}\left[  h_{y}\left(  y_{0}^{\mu}\right)  \widetilde
{y}_{0}\right] \nonumber\\
&  +\mathbb{E}%
{\displaystyle\int\nolimits_{0}^{T}}
\left[
{\displaystyle\int\nolimits_{U}}
l\left(  t,x_{t}^{\mu},y_{t}^{\mu},z_{t}^{\mu},a\right)  q_{t}\left(
da\right)  -%
{\displaystyle\int\nolimits_{U}}
l\left(  t,x_{t}^{\mu},y_{t}^{\mu},z_{t}^{\mu},a\right)  \mu_{t}\left(
da\right)  \right]  dt\\
&  +\mathbb{E}%
{\displaystyle\int\nolimits_{0}^{T}}
\left[
{\displaystyle\int\nolimits_{U}}
l_{x}\left(  t,x_{t}^{\mu},y_{t}^{\mu},z_{t}^{\mu},a\right)  \mu_{t}\left(
da\right)  \widetilde{x}_{t}+%
{\displaystyle\int\nolimits_{U}}
l_{y}\left(  t,x_{t}^{\mu},y_{t}^{\mu},z_{t}^{\mu},a\right)  \mu_{t}\left(
da\right)  \widetilde{y}_{t}\right]  dt\nonumber\\
&  +\mathbb{E}%
{\displaystyle\int\nolimits_{0}^{T}}
{\displaystyle\int\nolimits_{U}}
l_{z}\left(  t,x_{t}^{\mu},y_{t}^{\mu},z_{t}^{\mu},a\right)  \mu_{t}\left(
da\right)  \widetilde{z}_{t}dt.\nonumber
\end{align}

\end{lemma}

\begin{proof}
Let $\mu$ be an optimal relaxed control minimizing the cost $\mathcal{J}$ over
$\mathcal{R}$, then we get%
\begin{align*}
0  &  \leq\mathbb{E}\left[  g\left(  x_{T}^{\theta}\right)  -g\left(
x_{T}^{\mu}\right)  \right]  +\mathbb{E}\left[  h\left(  y_{0}^{\theta
}\right)  -h\left(  y_{0}^{\mu}\right)  \right] \\
&  +\mathbb{E}%
{\displaystyle\int\nolimits_{0}^{T}}
\left[
{\displaystyle\int\nolimits_{U}}
l\left(  t,x_{t}^{\theta},y_{t}^{\theta},z_{t}^{\theta},a\right)  \mu
_{t}^{\theta}\left(  da\right)  -%
{\displaystyle\int\nolimits_{U}}
l\left(  t,x_{t}^{\mu},y_{t}^{\mu},z_{t}^{\mu},a\right)  \mu_{t}\left(
da\right)  \right]  dt\\
&  =\mathbb{E}\left[  g\left(  x_{T}^{\theta}\right)  -g\left(  x_{T}^{\mu
}\right)  \right]  +\mathbb{E}\left[  h\left(  y_{0}^{\theta}\right)
-h\left(  y_{0}^{\mu}\right)  \right] \\
&  +\mathbb{E}%
{\displaystyle\int\nolimits_{0}^{T}}
\left[
{\displaystyle\int\nolimits_{U}}
l\left(  t,x_{t}^{\theta},y_{t}^{\theta},z_{t}^{\theta},a\right)  \mu
_{t}^{\theta}\left(  da\right)  -%
{\displaystyle\int\nolimits_{U}}
l\left(  t,x_{t}^{\theta},y_{t}^{\theta},z_{t}^{\theta},a\right)  \mu
_{t}\left(  da\right)  \right]  dt\\
&  +\mathbb{E}%
{\displaystyle\int\nolimits_{0}^{T}}
\left[
{\displaystyle\int\nolimits_{U}}
l\left(  t,x_{t}^{\theta},y_{t}^{\theta},z_{t}^{\theta},a\right)  \mu
_{t}\left(  da\right)  -%
{\displaystyle\int\nolimits_{U}}
l\left(  t,x_{t}^{\mu},y_{t}^{\mu},z_{t}^{\mu},a\right)  \mu_{t}\left(
da\right)  \right]  dt.
\end{align*}

From the definition of $\mu^{\theta}$, we get%
\begin{align*}
0  &  \leq\mathbb{E}\left[  g\left(  x_{T}^{\theta}\right)  -g\left(
x_{T}^{\mu}\right)  \right]  +\mathbb{E}\left[  h\left(  y_{0}^{\theta
}\right)  -h\left(  y_{0}^{\mu}\right)  \right] \\
&  +\theta\mathbb{E}%
{\displaystyle\int\nolimits_{0}^{T}}
\left[
{\displaystyle\int\nolimits_{U}}
l\left(  t,x_{t}^{\theta},y_{t}^{\theta},z_{t}^{\theta},a\right)  q_{t}\left(
da\right)  -%
{\displaystyle\int\nolimits_{U}}
l\left(  t,x_{t}^{\theta},y_{t}^{\theta},z_{t}^{\theta},a\right)  \mu
_{t}\left(  da\right)  \right]  dt\\
&  +\mathbb{E}%
{\displaystyle\int\nolimits_{0}^{T}}
{\displaystyle\int\nolimits_{U}}
\left[  l\left(  t,x_{t}^{\theta},y_{t}^{\theta},z_{t}^{\theta},a\right)
-l\left(  t,x_{t}^{\mu},y_{t}^{\mu},z_{t}^{\mu},a\right)  \right]  \mu
_{t}\left(  da\right)  dt.
\end{align*}

Then,%
\begin{align}
0  &  \leq\mathbb{E}%
{\displaystyle\int\nolimits_{0}^{1}}
\left[  g_{x}\left(  x_{T}^{\mu}+\lambda\theta\left(  \widetilde{x}_{T}%
+X_{T}^{\theta}\right)  \right)  \widetilde{x}_{T}\right]  d\lambda\\
&  +\mathbb{E}%
{\displaystyle\int\nolimits_{0}^{1}}
\left[  h_{y}\left(  y_{0}^{\mu}+\lambda\theta\left(  \widetilde{y}_{0}%
+Y_{0}^{\theta}\right)  \right)  \widetilde{y}_{0}\right]  d\lambda\nonumber\\
&  +\mathbb{E}%
{\displaystyle\int\nolimits_{0}^{T}}
{\displaystyle\int\nolimits_{0}^{1}}
{\displaystyle\int\nolimits_{U}}
\left[  l_{x}\left(  \Lambda_{t}^{\theta}\left(  a\right)  \right)
\widetilde{x}_{t}+l_{y}\left(  \Lambda_{t}^{\theta}\left(  a\right)  \right)
\widetilde{y}_{t}+l_{z}\left(  \Lambda_{t}^{\theta}\left(  a\right)  \right)
\widetilde{z}_{t}\right]  \mu_{t}\left(  da\right)  d\lambda dt\nonumber\\
&  +\mathbb{E}%
{\displaystyle\int\nolimits_{0}^{T}}
\left[
{\displaystyle\int\nolimits_{U}}
l\left(  t,x_{t}^{\mu},y_{t}^{\mu},z_{t}^{\mu},a\right)  q_{t}\left(
da\right)  -%
{\displaystyle\int\nolimits_{U}}
l\left(  t,x_{t}^{\mu},y_{t}^{\mu},z_{t}^{\mu},a\right)  \mu_{t}\left(
da\right)  \right]  dt+\rho_{t}^{\theta},\nonumber
\end{align}
where $\rho_{t}^{\theta}$ is given by%
\begin{align*}
&  \rho_{t}^{\theta}=\mathbb{E}%
{\displaystyle\int\nolimits_{0}^{1}}
\left[  g_{x}\left(  x_{T}^{\mu}+\lambda\theta\left(  \widetilde{x}_{T}%
+X_{T}^{\theta}\right)  \right)  X_{T}^{\theta}\right]  d\lambda\\
&  +\mathbb{E}%
{\displaystyle\int\nolimits_{0}^{1}}
\left[  h_{y}\left(  y_{0}^{\mu}+\lambda\theta\left(  \widetilde{y}_{0}%
+Y_{0}^{\theta}\right)  \right)  Y_{0}^{\theta}\right]  d\lambda\\
&  +\mathbb{E}%
{\displaystyle\int\nolimits_{0}^{T}}
{\displaystyle\int\nolimits_{0}^{1}}
{\displaystyle\int\nolimits_{U}}
\left[  l_{x}\left(  \Lambda_{t}^{\theta}\left(  a\right)  \right)  \left(
x_{t}^{\theta}-x_{t}^{\mu}\right)  +l_{y}\left(  \Lambda_{t}^{\theta}\left(
a\right)  \right)  \left(  y_{t}^{\theta}-y_{t}^{\mu}\right)  +l_{z}\left(
\Lambda_{t}^{\theta}\left(  a\right)  \right)  \left(  z_{t}^{\theta}%
-z_{t}^{\mu}\right)  \right]  q_{t}\left(  da\right)  d\lambda dt\\
&  +\mathbb{E}%
{\displaystyle\int\nolimits_{0}^{T}}
{\displaystyle\int\nolimits_{0}^{1}}
{\displaystyle\int\nolimits_{U}}
\left[  l_{x}\left(  \Lambda_{t}^{\theta}\left(  a\right)  \right)  \left(
x_{t}^{\theta}-x_{t}^{\mu}\right)  +l_{y}\left(  \Lambda_{t}^{\theta}\left(
a\right)  \right)  \left(  y_{t}^{\theta}-y_{t}^{\mu}\right)  +l_{z}\left(
\Lambda_{t}^{\theta}\left(  a\right)  \right)  \left(  z_{t}^{\theta}%
-z_{t}^{\mu}\right)  \right]  \mu_{t}\left(  da\right)  d\lambda dt\\
&  +\mathbb{E}%
{\displaystyle\int\nolimits_{0}^{T}}
{\displaystyle\int\nolimits_{0}^{1}}
{\displaystyle\int\nolimits_{U}}
\left[  l_{x}\left(  \Lambda_{t}^{\theta}\left(  a\right)  \right)
X_{T}^{\theta}+l_{y}\left(  \Lambda_{t}^{\theta}\left(  a\right)  \right)
Y_{t}^{\theta}+l_{z}\left(  \Lambda_{t}^{\theta}\left(  a\right)  \right)
Z_{t}^{\theta}\right]  \mu_{t}\left(  da\right)  d\lambda dt.
\end{align*}

Since the derivatives $g_{x},\ h_{y},\ l_{x},\ l_{y}$ and $l_{z}$ are bounded,
then by using the Cauchy-Schwartz inequality, we have%
\begin{align*}
&  \rho_{t}^{\theta}\leq C\left(  \mathbb{E}\left\vert X_{T}^{\theta
}\right\vert ^{2}\right)  ^{1/2}+C\left(  \mathbb{E}\left\vert Y_{0}^{\theta
}\right\vert ^{2}\right)  ^{1/2}\\
&  +C\left(  \mathbb{E}%
{\displaystyle\int\nolimits_{0}^{T}}
\left\vert x_{t}^{\theta}-x_{t}^{\mu}\right\vert ^{2}dt\right)  ^{1/2}%
+C\left(  \mathbb{E}%
{\displaystyle\int\nolimits_{0}^{T}}
\left\vert y_{t}^{\theta}-y_{t}^{\mu}\right\vert ^{2}dt\right)  ^{1/2}%
+C\left(  \mathbb{E}%
{\displaystyle\int\nolimits_{0}^{T}}
\left\vert z_{t}^{\theta}-z_{t}^{\mu}\right\vert ^{2}dt\right)  ^{1/2}\\
&  +C\left(  \mathbb{E}%
{\displaystyle\int\nolimits_{0}^{T}}
\left\vert X_{t}^{\theta}\right\vert ^{2}dt\right)  ^{1/2}+C\left(  \mathbb{E}%
{\displaystyle\int\nolimits_{0}^{T}}
\left\vert Y_{t}^{\theta}\right\vert ^{2}dt\right)  ^{1/2}+C\left(  \mathbb{E}%
{\displaystyle\int\nolimits_{0}^{T}}
\left\vert Z_{t}^{\theta}\right\vert ^{2}dt\right)  ^{1/2}.
\end{align*}

By using $\left(  8\right)  ,\ \left(  9\right)  ,\ \left(  10\right)
,\ \left(  17\right)  ,\ \left(  18\right)  $ and $\left(  19\right)  $, we
get%
\[
\underset{\theta\rightarrow0}{\lim}\rho_{t}^{\theta}=0.
\]

Since $g_{x},\ h_{y},\ l_{x},\ l_{y}$ and $l_{z}$ are continuous and bounded,
the proof is completed by letting $\theta$ go to $0$ in $\left(  29\right)  $.
\end{proof}

\subsection{Necessary optimality conditions for relaxed controls}

Starting from the variational inequality $\left(  28\right)  $, we can now
state necessary optimality conditions for the relaxed control problem
$\left\{  \left(  5\right)  ,\left(  6\right)  ,\left(  7\right)  \right\}  $
in the global form.

\begin{theorem}
(Necessary optimality conditions for relaxed controls). \textit{Let }$\mu
$\textit{\ be an optimal relaxed control minimizing the functional
}$\mathcal{J}$\textit{\ over }$\mathcal{R}$\textit{\ and }$\left(  x^{\mu
},y^{\mu},z^{\mu}\right)  $\textit{\ the solution of }$\left(  5\right)
$\textit{\ controlled by }$\mu$\textit{. }Then, there exist three adapted
processes $\left(  k^{\mu},p^{\mu},P^{\mu}\right)  $\textit{, uniqe solution
of the following FBSDE system (called adjoint equations)}%
\begin{equation}
\left\{
\begin{array}
[c]{ll}%
dk_{t}^{\mu}= & \mathcal{H}_{y}\left(  t,x_{t}^{\mu},y_{t}^{\mu},z_{t}^{\mu
},\mu_{t},k_{t}^{\mu},p_{t}^{\mu},P_{t}^{\mu}\right)  dt\\
& +\mathcal{H}_{z}\left(  t,x_{t}^{\mu},y_{t}^{\mu},z_{t}^{\mu},\mu_{t}%
,k_{t}^{\mu},p_{t}^{\mu},P_{t}^{\mu}\right)  dW_{t},\\
k_{0}^{\mu}= & h_{y}\left(  y_{0}^{\mu}\right) \\
dp_{t}^{\mu}= & -\mathcal{H}_{x}\left(  t,x_{t}^{\mu},y_{t}^{\mu},z_{t}^{\mu
},\mu_{t},k_{t}^{\mu},p_{t}^{\mu},P_{t}^{\mu}\right)  dt+P_{t}^{\mu}dW_{t},\\
p_{T}^{\mu}= & g_{x}\left(  x_{T}^{\mu}\right)  +\varphi_{x}\left(  x_{T}%
^{\mu}\right)  k_{T}^{\mu},
\end{array}
\right.
\end{equation}
such that for every $q_{t}\in\mathbb{P}\left(  U\right)  $%
\begin{equation}
\mathcal{H}\left(  t,x_{t}^{\mu},y_{t}^{\mu},z_{t}^{\mu},\mu_{t},p_{t}^{\mu
},k_{t}^{\mu},P_{t}^{\mu}\right)  \leq\mathcal{H}\left(  t,x_{t}^{\mu}%
,y_{t}^{\mu},z_{t}^{\mu},q_{t},k_{t}^{\mu},p_{t}^{\mu},P_{t}^{\mu}\right)
,\ ae\ ,as,
\end{equation}
where the Hamiltonian $\mathcal{H}$\ is defined from $\left[  0,T\right]
\times\mathbb{R}^{n}\times\mathbb{R}^{m}\times\mathcal{M}_{m\times d}\left(
\mathbb{R}\right)  \times\mathbb{P}\left(  U\right)  \times\mathbb{R}%
^{m}\times\mathbb{R}^{n}\times\mathcal{M}_{n\times d}\left(  \mathbb{R}%
\right)  $\ into $\mathbb{R}$\ by%
\begin{align*}
\mathcal{H}\left(  t,x,y,z,q,k,p,P\right)   &  =%
{\displaystyle\int\nolimits_{U}}
l\left(  t,x,y,z,a\right)  q_{t}\left(  da\right)  +p%
{\displaystyle\int\nolimits_{U}}
b\left(  t,x,a\right)  q_{t}\left(  da\right) \\
&  +P%
{\displaystyle\int\nolimits_{U}}
\sigma\left(  t,x,a\right)  q_{t}\left(  da\right)  +k%
{\displaystyle\int\nolimits_{U}}
f\left(  t,x,y,z,a\right)  q_{t}\left(  da\right)  .
\end{align*}

\end{theorem}

\begin{proof}
Since$\ k_{0}^{\mu}=h_{y}(y_{0}^{\mu})$ and $p_{T}^{\mu}=g_{x}\left(
x_{T}^{\mu}\right)  +\varphi_{x}\left(  x_{T}^{\mu}\right)  k_{T}^{\mu}$, then
$\left(  28\right)  $ becomes%
\begin{align}
0  &  \leq\mathbb{E}\left[  p_{T}^{\mu}\widetilde{x}_{T}\right]
+\mathbb{E}\left[  k_{0}^{\mu}\widetilde{y}_{0}\right]  -\mathbb{E}\left[
\varphi_{x}\left(  x_{T}^{\mu}\right)  k_{T}^{\mu}\right]  +\mathbb{E}%
{\displaystyle\int\nolimits_{0}^{T}}
{\displaystyle\int\nolimits_{U}}
l_{x}\left(  t,x_{t}^{\mu},y_{t}^{\mu},z_{t}^{\mu},a\right)  \widetilde{x}%
_{t}\mu_{t}\left(  da\right)  dt\nonumber\\
&  +\mathbb{E}%
{\displaystyle\int\nolimits_{0}^{T}}
{\displaystyle\int\nolimits_{U}}
\left[  l_{y}\left(  t,x_{t}^{\mu},y_{t}^{\mu},z_{t}^{\mu},a\right)
\widetilde{y}_{t}+l_{z}\left(  t,x_{t}^{\mu},y_{t}^{\mu},z_{t}^{\mu},a\right)
\widetilde{z}_{t}\right]  \mu_{t}\left(  da\right)  dt\\
&  +\mathbb{E}%
{\displaystyle\int\nolimits_{0}^{T}}
\left[
{\displaystyle\int\nolimits_{U}}
l\left(  t,x_{t}^{\mu},y_{t}^{\mu},z_{t}^{\mu},a\right)  q_{t}\left(
da\right)  -%
{\displaystyle\int\nolimits_{U}}
l\left(  t,x_{t}^{\mu},y_{t}^{\mu},z_{t}^{\mu},a\right)  \mu_{t}\left(
da\right)  \right]  dt.\nonumber
\end{align}

By applying It\^{o}'s formula to $\left(  p_{t}^{\mu}\widetilde{x}_{t}\right)
$ and $\left(  k_{t}^{\mu}\widetilde{y}_{t}\right)  $, we have%
\begin{align*}
\mathbb{E}\left[  p_{T}^{\mu}\widetilde{x}_{T}\right]   &  =-\mathbb{E}%
{\displaystyle\int\nolimits_{0}^{T}}
\left[
{\displaystyle\int\nolimits_{U}}
f_{x}\left(  t,x_{t}^{\mu},y_{t}^{\mu},z_{t}^{\mu},a\right)  \mu_{t}\left(
da\right)  k_{t}^{\mu}+%
{\displaystyle\int\nolimits_{U}}
l_{x}\left(  t,x_{t}^{\mu},y_{t}^{\mu},z_{t}^{\mu},a\right)  \mu_{t}\left(
da\right)  \right]  \widetilde{x}_{t}dt\\
&  +\mathbb{E}%
{\displaystyle\int\nolimits_{0}^{T}}
p_{t}^{\mu}\left[
{\displaystyle\int\nolimits_{U}}
b\left(  t,x_{t}^{\mu},a\right)  q_{t}\left(  da\right)  -%
{\displaystyle\int\nolimits_{U}}
b\left(  t,x_{t}^{\mu},a\right)  \mu_{t}\left(  da\right)  \right]  dt\\
&  +\mathbb{E}%
{\displaystyle\int\nolimits_{0}^{T}}
P_{t}^{\mu}\left[
{\displaystyle\int\nolimits_{U}}
\sigma\left(  t,x_{t}^{\mu},a\right)  q_{t}\left(  da\right)  -%
{\displaystyle\int\nolimits_{U}}
\sigma\left(  t,x_{t}^{\mu},a\right)  \mu_{t}\left(  da\right)  \right]  dt.
\end{align*}%
\begin{align*}
\mathbb{E}\left[  k_{0}^{\mu}\widetilde{y}_{0}\right]   &  =\mathbb{E}\left[
k_{T}^{\mu}\widetilde{y}_{T}\right]  -\mathbb{E}\left[
{\displaystyle\int\nolimits_{0}^{T}}
{\displaystyle\int\nolimits_{U}}
l_{y}\left(  t,x_{t}^{\mu},y_{t}^{\mu},z_{t}^{\mu},a\right)  \mu_{t}\left(
da\right)  \widetilde{y}_{t}+%
{\displaystyle\int\nolimits_{U}}
f_{x}\left(  t,x_{t}^{\mu},y_{t}^{\mu},z_{t}^{\mu},a\right)  \mu_{t}\left(
da\right)  \widetilde{x}_{t}k_{t}^{\mu}\right]  dt\\
&  +\mathbb{E}%
{\displaystyle\int\nolimits_{0}^{T}}
k_{t}^{\mu}\left[
{\displaystyle\int\nolimits_{U}}
f\left(  t,x_{t}^{\mu},y_{t}^{\mu},z_{t}^{\mu},a\right)  q_{t}\left(
da\right)  -%
{\displaystyle\int\nolimits_{U}}
f\left(  t,x_{t}^{\mu},y_{t}^{\mu},z_{t}^{\mu},a\right)  \mu_{t}\left(
da\right)  \right]  dt\\
&  -\mathbb{E}%
{\displaystyle\int\nolimits_{0}^{T}}
{\displaystyle\int\nolimits_{U}}
l_{z}\left(  t,x_{t}^{\mu},y_{t}^{\mu},z_{t}^{\mu},a\right)  \mu_{t}\left(
da\right)  \widetilde{z}_{t}dt.
\end{align*}

Then for every $q\in\mathcal{R}$, $\left(  32\right)  $ becomes%
\[
0\leq\mathbb{E}%
{\displaystyle\int\nolimits_{0}^{T}}
\left[  \mathcal{H}\left(  t,x_{t}^{\mu},y_{t}^{\mu},z_{t}^{\mu},q_{t}%
,k_{t}^{\mu},p_{t}^{\mu},P_{t}^{\mu}\right)  -\mathcal{H}\left(  t,x_{t}^{\mu
},y_{t}^{\mu},z_{t}^{\mu},\mu_{t},k_{t}^{\mu},p_{t}^{\mu},P_{t}^{\mu}\right)
\right]  dt.
\]

Now, let $q\in\mathcal{R}$ and $F$ be an arbitrary element of the $\sigma
$-algebra $\mathcal{F}_{t}$, and set%
\[
\pi_{t}=q_{t}\mathbf{1}_{F}+\mu_{t}\mathbf{1}_{%
\Omega
-F}.
\]

It is obvious that $\pi$ is an admissible relaxed control.

Applying the above inequality with $\pi$, we get
\[
0\leq\mathbb{E}[\mathbf{1}_{F}\left\{  \mathcal{H}\left(  t,x_{t}^{\mu}%
,y_{t}^{\mu},z_{t}^{\mu},q_{t},k_{t}^{\mu},p_{t}^{\mu},P_{t}^{\mu}\right)
-\mathcal{H}\left(  t,x_{t}^{\mu},y_{t}^{\mu},z_{t}^{\mu},\mu_{t},k_{t}^{\mu
},p_{t}^{\mu},P_{t}^{\mu}\right)  \right\}  ],\ \forall F\in\mathcal{F}_{t}.
\]

Which implies that
\[
0\leq\mathbb{E}[\mathcal{H}\left(  t,x_{t}^{\mu},y_{t}^{\mu},z_{t}^{\mu}%
,q_{t},k_{t}^{\mu},p_{t}^{\mu},P_{t}^{\mu}\right)  -\mathcal{H}\left(
t,x_{t}^{\mu},y_{t}^{\mu},z_{t}^{\mu},\mu_{t},k_{t}^{\mu},p_{t}^{\mu}%
,P_{t}^{\mu}\right)  \ /\ \mathcal{F}_{t}].
\]

The quantity inside the conditional expectation is $\mathcal{F}_{t}%
$-measurable, and thus the result follows immediately.
\end{proof}

\subsection{Sufficient optimality conditions for relaxed controls}

In this subsection, we study when necessary optimality conditions $\left(
31\right)  $ becomes sufficient. For any $q\in\mathcal{R}$, we denote by
$\left(  x^{q},y^{q},z^{q}\right)  $ the solution of equation $\left(
5\right)  $ controlled by $q.$

\begin{theorem}
(Sufficient optimality conditions for relaxed controls). Assume that the
functions $g,\ h$ and $\left(  x,y,z\right)  \longmapsto\mathcal{H}\left(
t,x,y,z,q,k,p,P\right)  $ are convex, and for any $q\in\mathcal{R}$,
$y_{T}^{q}=\xi,$ where $\xi$\ is an $m$-dimensional $\mathcal{F}_{T}%
$-measurable random variable such that%
\[
\mathbb{E}\left\vert \xi\right\vert ^{2}<\infty.
\]

Then, $\mu$ is an optimal solution of the relaxed control problem $\left\{
\left(  5\right)  ,\left(  6\right)  ,\left(  7\right)  \right\}  $, if it
satisfies $\left(  31\right)  .$
\end{theorem}

\begin{proof}
Let $\mu$ be an arbitrary element of $\mathcal{R}$ (candidate to be optimal).
For any $q\in\mathcal{R}$, we have%
\begin{align*}
\mathcal{J}\left(  q\right)  -\mathcal{J}\left(  \mu\right)   &
=\mathbb{E}\left[  g\left(  x_{T}^{q}\right)  -g\left(  x_{T}^{\mu}\right)
\right]  +\mathbb{E}\left[  h\left(  y_{0}^{q}\right)  -h\left(  y_{0}^{\mu
}\right)  \right] \\
&  +\mathbb{E}%
{\displaystyle\int\nolimits_{0}^{T}}
\left[
{\displaystyle\int\nolimits_{U}}
l\left(  t,x_{t}^{q},y_{t}^{q},z_{t}^{q},a\right)  q_{t}\left(  da\right)  -%
{\displaystyle\int\nolimits_{U}}
l\left(  t,x_{t}^{\mu},y_{t}^{\mu},z_{t}^{\mu},a\right)  \mu_{t}\left(
da\right)  \right]  dt.
\end{align*}

Since $g$ and $h$ are convex, then%
\begin{align*}
g\left(  x_{T}^{q}\right)  -g\left(  x_{T}^{\mu}\right)   &  \geq g_{x}\left(
x_{T}^{\mu}\right)  \left(  x_{T}^{q}-x_{T}^{\mu}\right)  ,\\
h\left(  y_{0}^{q}\right)  -h\left(  y_{0}^{\mu}\right)   &  \geq h_{y}\left(
y_{0}^{\mu}\right)  \left(  y_{0}^{q}-y_{0}^{\mu}\right)  .
\end{align*}

Thus,%
\begin{align*}
\mathcal{J}\left(  q\right)  -\mathcal{J}\left(  \mu\right)   &
\geq\mathbb{E}\left[  g_{x}\left(  x_{T}^{\mu}\right)  \left(  x_{T}^{q}%
-x_{T}^{\mu}\right)  \right]  +\mathbb{E}\left[  h_{y}\left(  y_{0}^{\mu
}\right)  \left(  y_{0}^{q}-y_{0}^{\mu}\right)  \right] \\
&  +\mathbb{E}%
{\displaystyle\int\nolimits_{0}^{T}}
\left[
{\displaystyle\int\nolimits_{U}}
l\left(  t,x_{t}^{q},y_{t}^{q},z_{t}^{q},a\right)  q_{t}\left(  da\right)  -%
{\displaystyle\int\nolimits_{U}}
l\left(  t,x_{t}^{\mu},y_{t}^{\mu},z_{t}^{\mu},a\right)  \mu_{t}\left(
da\right)  \right]  dt.
\end{align*}

We remark from $\left(  30\right)  $ that
\begin{align*}
p_{T}^{\mu}  &  =g_{x}\left(  x_{T}^{\mu}\right)  ,\\
k_{0}^{\mu}  &  =h_{y}\left(  y_{0}^{\mu}\right)  .
\end{align*}

Then, we have%
\begin{align*}
\mathcal{J}\left(  q\right)  -\mathcal{J}\left(  \mu\right)   &
\geq\mathbb{E}\left[  p_{T}^{\mu}\left(  x_{T}^{q}-x_{T}^{\mu}\right)
\right]  +\mathbb{E}\left[  k_{0}^{\mu}\left(  y_{0}^{q}-y_{0}^{\mu}\right)
\right] \\
&  +\mathbb{E}%
{\displaystyle\int\nolimits_{0}^{T}}
\left[
{\displaystyle\int\nolimits_{U}}
l\left(  t,x_{t}^{q},y_{t}^{q},z_{t}^{q},a\right)  q_{t}\left(  da\right)  -%
{\displaystyle\int\nolimits_{U}}
l\left(  t,x_{t}^{\mu},y_{t}^{\mu},z_{t}^{\mu},a\right)  \mu_{t}\left(
da\right)  \right]  dt.
\end{align*}

By applying It\^{o}'s formula respectively to $p_{t}^{\mu}\left(  x_{t}%
^{q}-x_{t}^{\mu}\right)  $ and $k_{t}^{\mu}\left(  y_{t}^{q}-y_{t}^{\mu
}\right)  $, we obtain%
\begin{align*}
\mathbb{E}\left[  p_{T}^{\mu}\left(  x_{T}^{q}-x_{T}^{\mu}\right)  \right]
&  =-\mathbb{E}%
{\displaystyle\int\nolimits_{0}^{T}}
\mathcal{H}_{x}\left(  t,x_{t}^{\mu},y_{t}^{\mu},z_{t}^{\mu},\mu_{t}%
,k_{t}^{\mu},p_{t}^{\mu},P_{t}^{\mu}\right)  \left(  x_{t}^{q}-x_{t}^{\mu
}\right)  dt\\
&  +\mathbb{E}%
{\displaystyle\int\nolimits_{0}^{T}}
p_{t}^{\mu}\left[
{\displaystyle\int\nolimits_{U}}
b\left(  t,x_{t}^{q},a\right)  q_{t}\left(  da\right)  -%
{\displaystyle\int\nolimits_{U}}
b\left(  t,x_{t}^{\mu},a\right)  \mu_{t}\left(  da\right)  \right]  dt\\
&  +\mathbb{E}%
{\displaystyle\int\nolimits_{0}^{T}}
P_{t}^{\mu}\left[
{\displaystyle\int\nolimits_{U}}
\sigma\left(  t,x_{t}^{q},a\right)  q_{t}\left(  da\right)  -%
{\displaystyle\int\nolimits_{U}}
\sigma\left(  t,x_{t}^{\mu},a\right)  \mu_{t}\left(  da\right)  \right]  dt,
\end{align*}%
\begin{align*}
\mathbb{E}\left[  k_{0}^{\mu}\left(  y_{0}^{q}-y_{0}^{\mu}\right)  \right]
&  =-\mathbb{E}%
{\displaystyle\int\nolimits_{0}^{T}}
\mathcal{H}_{y}\left(  t,x_{t}^{\mu},y_{t}^{\mu},z_{t}^{\mu},\mu_{t}%
,k_{t}^{\mu},p_{t}^{\mu},P_{t}^{\mu}\right)  \left(  y_{t}^{q}-y_{t}^{\mu
}\right)  dt\\
&  +\mathbb{E}%
{\displaystyle\int\nolimits_{0}^{T}}
k_{t}^{\mu}\left[
{\displaystyle\int\nolimits_{U}}
f\left(  t,x_{t}^{q},y_{t}^{q},z_{t}^{q},a\right)  q_{t}\left(  da\right)  -%
{\displaystyle\int\nolimits_{U}}
f\left(  t,x_{t}^{\mu},y_{t}^{\mu},z_{t}^{\mu},a\right)  \mu_{t}\left(
da\right)  \right]  dt\\
&  -\mathbb{E}%
{\displaystyle\int\nolimits_{0}^{T}}
\mathcal{H}_{z}\left(  t,x_{t}^{\mu},y_{t}^{\mu},z_{t}^{\mu},\mu_{t}%
,k_{t}^{\mu},p_{t}^{\mu},P_{t}^{\mu}\right)  \left(  z_{t}^{q}-z_{t}^{\mu
}\right)  dt.
\end{align*}

Then,%
\begin{align}
&  \mathcal{J}\left(  q\right)  -\mathcal{J}\left(  \mu\right) \\
&  \geq\mathbb{E}%
{\displaystyle\int\nolimits_{0}^{T}}
\left[  \mathcal{H}\left(  t,x_{t}^{q},y_{t}^{q},z_{t}^{q},q_{t},k_{t}^{\mu
},p_{t}^{\mu},P_{t}^{\mu}\right)  -\mathcal{H}\left(  t,x_{t}^{\mu},y_{t}%
^{\mu},z_{t}^{\mu},\mu_{t},k_{t}^{\mu},p_{t}^{\mu},P_{t}^{\mu}\right)
\right]  dt\nonumber\\
&  -\mathbb{E}%
{\displaystyle\int\nolimits_{0}^{T}}
\mathcal{H}_{x}\left(  t,x_{t}^{\mu},y_{t}^{\mu},z_{t}^{\mu},\mu_{t}%
,k_{t}^{\mu},p_{t}^{\mu},P_{t}^{\mu}\right)  \left(  x_{t}^{q}-x_{t}^{\mu
}\right)  dt\nonumber\\
&  -\mathbb{E}%
{\displaystyle\int\nolimits_{0}^{T}}
\mathcal{H}_{y}\left(  t,x_{t}^{\mu},y_{t}^{\mu},z_{t}^{\mu},\mu_{t}%
,k_{t}^{\mu},p_{t}^{\mu},P_{t}^{\mu}\right)  \left(  y_{t}^{q}-y_{t}^{\mu
}\right)  dt\nonumber\\
&  -\mathbb{E}%
{\displaystyle\int\nolimits_{0}^{T}}
\mathcal{H}_{z}\left(  t,x_{t}^{\mu},y_{t}^{\mu},z_{t}^{\mu},\mu_{t}%
,k_{t}^{\mu},p_{t}^{\mu},P_{t}^{\mu}\right)  \left(  z_{t}^{q}-z_{t}^{\mu
}\right)  dt.\nonumber
\end{align}

Since $\mathcal{H}$ is convex in $\left(  x,y,z\right)  $ and linear in $\mu$,
then by using the Clarke generalized gradient of $\mathcal{H}$ evaluated at
$\left(  x_{t},y_{t},z_{t},\mu_{t}\right)  $ and the necessary optimality
conditions $\left(  31\right)  $, it follows by $\left[
50\text{,\ Lemmas\ 2.2 and 2.3}\right]  $ that%
\begin{align*}
&  \mathcal{H}\left(  t,x_{t}^{q},y_{t}^{q},z_{t}^{q},q_{t},k_{t}^{\mu}%
,p_{t}^{\mu},P_{t}^{\mu}\right)  -\mathcal{H}\left(  t,x_{t}^{\mu},y_{t}^{\mu
},z_{t}^{\mu},\mu_{t},k_{t}^{\mu},p_{t}^{\mu},P_{t}^{\mu}\right) \\
&  \geq\mathcal{H}_{x}\left(  t,x_{t}^{\mu},y_{t}^{\mu},z_{t}^{\mu},\mu
_{t},k_{t}^{\mu},p_{t}^{\mu},P_{t}^{\mu}\right)  \left(  x_{t}^{q}-x_{t}^{\mu
}\right)  +\mathcal{H}_{y}\left(  t,x_{t}^{\mu},y_{t}^{\mu},z_{t}^{\mu}%
,\mu_{t},k_{t}^{\mu},p_{t}^{\mu},P_{t}^{\mu}\right)  \left(  y_{t}^{q}%
-y_{t}^{\mu}\right) \\
&  +\mathcal{H}_{z}\left(  t,x_{t}^{\mu},y_{t}^{\mu},z_{t}^{\mu},\mu_{t}%
,k_{t}^{\mu},p_{t}^{\mu},P_{t}^{\mu}\right)  \left(  z_{t}^{q}-z_{t}^{\mu
}\right)  .
\end{align*}

Or equivalently,%
\begin{align*}
0  &  \leq\mathcal{H}\left(  t,x_{t}^{q},y_{t}^{q},z_{t}^{q},q_{t},k_{t}^{\mu
},p_{t}^{\mu},P_{t}^{\mu}\right)  -\mathcal{H}\left(  t,x_{t}^{\mu},y_{t}%
^{\mu},z_{t}^{\mu},\mu_{t},k_{t}^{\mu},p_{t}^{\mu},P_{t}^{\mu}\right) \\
&  -\mathcal{H}_{x}\left(  t,x_{t}^{\mu},y_{t}^{\mu},z_{t}^{\mu},\mu_{t}%
,k_{t}^{\mu},p_{t}^{\mu},P_{t}^{\mu}\right)  \left(  x_{t}^{q}-x_{t}^{\mu
}\right) \\
&  -\mathcal{H}_{y}\left(  t,x_{t}^{\mu},y_{t}^{\mu},z_{t}^{\mu},\mu_{t}%
,k_{t}^{\mu},p_{t}^{\mu},P_{t}^{\mu}\right)  \left(  y_{t}^{q}-y_{t}^{\mu
}\right) \\
&  -\mathcal{H}_{z}\left(  t,x_{t}^{\mu},y_{t}^{\mu},z_{t}^{\mu},\mu_{t}%
,k_{t}^{\mu},p_{t}^{\mu},P_{t}^{\mu}\right)  \left(  z_{t}^{q}-z_{t}^{\mu
}\right)  .
\end{align*}

Then from $\left(  33\right)  $, we get%
\[
\mathcal{J}\left(  q\right)  -\mathcal{J}\left(  \mu\right)  \geq0.
\]

The theorem is proved.
\end{proof}

\section{Optimality conditions for strict controls}

In this section, we study the strict control problem $\left\{  \left(
1\right)  ,\left(  2\right)  ,\left(  3\right)  \right\}  $ and from the
results of section 3, we derive the optimality conditions for strict controls.

\ 

Throughout this section and in addition to the assumptions $\left(  4\right)
$, we suppose that
\begin{align}
&  U\text{ is compact.}\\
&  b,\ \sigma,\ f\text{ and }l\text{ are\ bounded.}%
\end{align}

Consider the following subset of $\mathcal{R}$%
\[
\delta\left(  \mathcal{U}\right)  =\left\{  q\in\mathcal{R}\text{
\ /\ \ }q=\delta_{v}\ \ ;\ \ v\in\mathcal{U}\right\}  .
\]

The set $\delta\left(  \mathcal{U}\right)  $ is the collection of all relaxed
controls in the form of Dirac measure charging a strict control.

Denote by $\delta\left(  U\right)  $ the action set of all relaxed controls in
$\delta\left(  \mathcal{U}\right)  $.

If $q\in\delta\left(  \mathcal{U}\right)  $, then $q=\delta_{v}$ with
$v\in\mathcal{U}$. In this case we have for each $t$, $q_{t}\in\delta\left(
U\right)  $ and $q_{t}=\delta_{v_{t}}$.

\ 

We equipped $\mathbb{P}\left(  U\right)  $ with the topology of stable
convergence. Since $U$ is compact, then with this topology $\mathbb{P}\left(
U\right)  $\ is a compact metrizable space. The stable convergence is required
for bounded measurable functions $f\left(  t,a\right)  $\ such that for each
fixed $t\in\left[  0,T\right]  $, $f\left(  t,.\right)  $\ is continuous
(Instead of functions bounded and continuous with respect to the pair $\left(
t,a\right)  $ for the weak topology). The space $\mathbb{P}\left(  U\right)
$\ is equipped with its Borel $\sigma$-field, which is the smallest $\sigma
$-field such that the mapping $q\longmapsto%
{\displaystyle\int}
f\left(  s,a\right)  q\left(  ds,da\right)  $\ are measurable for any bounded
measurable function $f$, continuous with respect to $a.\ $For more details,
see Jacod and Memin $\left[  29\right]  $ and El Karoui et al $\left[
16\right]  $.

This allows us to summarize some of lemmas that we will be used in the sequel.\ 

\begin{lemma}
(Chattering\thinspace\thinspace Lemma).\thinspace\textit{Let }$q$\textit{\ be
a predictable process with values in the space of probability measures on }%
$U$\textit{. Then there exists a sequence of predictable processes }$\left(
u^{n}\right)  _{n}$\textit{\ with values in }$U$\textit{\ such that }%
\begin{equation}
dtq_{t}^{n}\left(  da\right)  =dt\delta_{u_{t}^{n}}\left(  da\right)
\underset{n\longrightarrow\infty}{\longrightarrow}dtq_{t}\left(  da\right)
\text{ stably},\text{\textit{\ \ }}\mathcal{P}-a.s.
\end{equation}
where $\delta_{u_{t}^{n}}$ is the Dirac measure concentrated at a single point
$u_{t}^{n}$ of $U$.
\end{lemma}

\begin{proof}
See El Karoui et al $\left[  16\right]  $.
\end{proof}

\begin{lemma}
Let $q$ be a relaxed control and $\left(  u^{n}\right)  _{n}$ be a sequence of
strict controls such that $\left(  36\right)  $ holds. Then for any bounded
measurable function $f:\left[  0,T\right]  \times U\rightarrow\mathbb{R}$,
such that for each fixed $t\in\left[  0,T\right]  $, $f\left(  t,.\right)  $
is continuous, we have%
\begin{equation}%
{\displaystyle\int\nolimits_{U}}
f\left(  t,a\right)  \delta_{u_{t}^{n}}\left(  da\right)  \underset
{n\longrightarrow\infty}{\longrightarrow}%
{\displaystyle\int\nolimits_{U}}
f\left(  t,a\right)  q_{t}\left(  da\right)  \ ;\ dt-a.e
\end{equation}

\end{lemma}

\begin{proof}
By $\left(  36\right)  $ and the definition of the stable convergence (see
Jacod-Memin $\left[  29,\ \text{definition 1.1, page 529}\right]  $, we have%
\[%
{\displaystyle\int\nolimits_{0}^{T}}
{\displaystyle\int\nolimits_{U}}
f\left(  t,a\right)  \delta_{u_{t}^{n}}\left(  da\right)  dt\underset
{n\longrightarrow\infty}{\longrightarrow}%
{\displaystyle\int\nolimits_{0}^{T}}
{\displaystyle\int\nolimits_{U}}
f\left(  t,a\right)  q_{t}\left(  da\right)  dt.
\]

Put%
\[
g\left(  s,a\right)  =1_{\left[  0,t\right]  }\left(  s\right)  f\left(
s,a\right)  .
\]

It's clear that $g$ is bounded, measurable and continuous with respect to $a$.
Then%
\[%
{\displaystyle\int\nolimits_{0}^{T}}
{\displaystyle\int\nolimits_{U}}
g\left(  s,a\right)  \delta_{u_{s}^{n}}\left(  da\right)  ds\underset
{n\longrightarrow\infty}{\longrightarrow}%
{\displaystyle\int\nolimits_{0}^{T}}
{\displaystyle\int\nolimits_{U}}
g\left(  s,a\right)  q_{s}\left(  da\right)  ds.
\]

By replacing $g\left(  s,a\right)  $ by its value, we have%
\[%
{\displaystyle\int\nolimits_{0}^{t}}
{\displaystyle\int\nolimits_{U}}
f\left(  s,a\right)  \delta_{u_{s}^{n}}\left(  da\right)  ds\underset
{n\longrightarrow\infty}{\longrightarrow}%
{\displaystyle\int\nolimits_{0}^{t}}
{\displaystyle\int\nolimits_{U}}
f\left(  s,a\right)  q_{s}\left(  da\right)  ds.
\]

The set $\left\{  \left(  s,t\right)  \ ;\ 0\leq s\leq t\leq T\right\}  $
generate $\mathcal{B}_{\left[  0,T\right]  }$. Then, for every $B\in
\mathcal{B}_{\left[  0,T\right]  }$ we have%
\[%
{\displaystyle\int\nolimits_{B}}
{\displaystyle\int\nolimits_{U}}
f\left(  s,a\right)  \delta_{u_{s}^{n}}\left(  da\right)  ds\underset
{n\longrightarrow\infty}{\longrightarrow}%
{\displaystyle\int\nolimits_{B}}
{\displaystyle\int\nolimits_{U}}
f\left(  s,a\right)  q_{s}\left(  da\right)  ds.
\]

This implies that%
\[%
{\displaystyle\int\nolimits_{U}}
f\left(  s,a\right)  \delta_{u_{s}^{n}}\left(  da\right)  \underset
{n\longrightarrow\infty}{\longrightarrow}%
{\displaystyle\int\nolimits_{U}}
f\left(  s,a\right)  q_{s}\left(  da\right)  \ ,\ \ dt-a.e.
\]

The lemma is proved.
\end{proof}

\ 

The next lemma gives the stability of the controlled FBSDE with respect to the
control variable.

\begin{lemma}
Let $q\in\mathcal{R}$ be a relaxed control and $\left(  x^{q},y^{q}%
,z^{q}\right)  $\textit{\ the corresponding trajectory. Then there exists a
sequence }$\left(  u^{n}\right)  _{n}\subset\mathcal{U}$ such that
\begin{align}
\underset{n\rightarrow\infty}{\lim}\mathbb{E}\left[  \underset{t\in\left[
0,T\right]  }{\sup}\left\vert x_{t}^{n}-x_{t}^{q}\right\vert ^{2}\right]   &
=0,\\
\underset{n\rightarrow\infty}{\lim}\mathbb{E}\left[  \underset{t\in\left[
0,T\right]  }{\sup}\left\vert y_{t}^{n}-y_{t}^{q}\right\vert ^{2}\right]   &
=0,\\
\underset{n\rightarrow\infty}{\lim}%
{\displaystyle\int\nolimits_{0}^{T}}
\mathbb{E}\left\vert z_{t}^{n}-z_{t}^{q}\right\vert ^{2}dt  &  =0,
\end{align}%
\begin{equation}
\underset{n\rightarrow\infty}{\lim}J\left(  u^{n}\right)  =\mathcal{J}\left(
q\right)  .
\end{equation}
where $\left(  x^{n},y^{n},z^{n}\right)  $ denotes the solution of equation
$\left(  1\right)  $ associated with $u^{n}.$
\end{lemma}

\begin{proof}
Proof of\textbf{ }$\left(  38\right)  $. We have%
\begin{align*}
\mathbb{E}\left\vert x_{t}^{n}-x_{t}^{q}\right\vert ^{2}  &  \leq C%
{\displaystyle\int\nolimits_{0}^{t}}
\mathbb{E}\left\vert b\left(  s,x_{s}^{n},u_{s}^{n}\right)  -\int_{U}b\left(
s,x_{s}^{q},a\right)  q_{s}\left(  da\right)  \right\vert ^{2}ds\\
&  +C%
{\displaystyle\int\nolimits_{0}^{t}}
\mathbb{E}\left\vert \sigma\left(  s,x_{s}^{n},u_{s}^{n}\right)  -\int
_{U}\sigma\left(  s,x_{s}^{q},a\right)  q_{s}\left(  da\right)  \right\vert
^{2}ds\\
&  \leq C%
{\displaystyle\int\nolimits_{0}^{t}}
\mathbb{E}\left\vert b\left(  s,x_{s}^{n},u_{s}^{n}\right)  -b\left(
s,x_{s}^{q},u_{s}^{n}\right)  \right\vert ^{2}ds\\
&  +C%
{\displaystyle\int\nolimits_{0}^{t}}
\mathbb{E}\left\vert b\left(  s,x_{s}^{q},u_{s}^{n}\right)  -\int_{U}b\left(
s,x_{s}^{q},a\right)  q_{s}\left(  da\right)  \right\vert ^{2}ds\\
&  +C%
{\displaystyle\int\nolimits_{0}^{t}}
\mathbb{E}\left\vert \sigma\left(  s,x_{s}^{n},u_{s}^{n}\right)
-\sigma\left(  s,x_{s}^{q},u_{s}^{n}\right)  \right\vert ^{2}ds\\
&  +C%
{\displaystyle\int\nolimits_{0}^{t}}
\mathbb{E}\left\vert \sigma\left(  s,x_{s}^{q},u_{s}^{n}\right)  -\int
_{U}\sigma\left(  s,x_{s}^{q},a\right)  q_{s}\left(  da\right)  \right\vert
^{2}ds
\end{align*}

Since $b$ and $\sigma$ are uniformly Lipschitz with respect to $x$, then%
\begin{align*}
\mathbb{E}\left\vert x_{t}^{n}-x_{t}^{q}\right\vert ^{2}  &  \leq C%
{\displaystyle\int\nolimits_{0}^{t}}
\mathbb{E}\left\vert x_{s}^{n}-x_{s}^{q}\right\vert ^{2}ds\\
&  +C%
{\displaystyle\int\nolimits_{0}^{t}}
\mathbb{E}\left\vert b\left(  s,x_{s}^{q},u_{s}^{n}\right)  -\int_{U}b\left(
s,x_{s}^{q},a\right)  q_{s}\left(  da\right)  \right\vert ^{2}ds\\
&  +C%
{\displaystyle\int\nolimits_{0}^{t}}
\mathbb{E}\left\vert \int_{U}\sigma\left(  s,x_{s}^{q},a\right)  \delta
_{u_{s}^{n}}\left(  da\right)  -\int_{U}\sigma\left(  s,x_{s}^{q},a\right)
q_{s}\left(  da\right)  \right\vert ^{2}ds
\end{align*}

Since $b$ and $\sigma$ are bounded, measurable and continuous with respect to
$a$, then by $\left(  37\right)  $ and the dominated convergence theorem, the
second and third terms in the right hand side of the above inequality tend to
zero as $n$ tends to infinity. We conclude then by using Gronwall's lemma and
Bukholder-Davis-Gundy inequality.

\ 

ii) Proof of $\left(  39\right)  $\textbf{\ }and\textbf{\ }$\left(  40\right)
$\textbf{.}

We have%
\[
\left\{
\begin{array}
[c]{ll}%
d\left(  y_{t}^{n}-y_{t}^{q}\right)  = & -\left[  f\left(  t,x_{t}^{n}%
,y_{t}^{n},z_{t}^{n},u_{t}^{n}\right)  -f\left(  t,x_{t}^{n},y_{t}^{q}%
,z_{t}^{q},u_{t}^{n}\right)  \right]  dt\\
& -\left[  f\left(  t,x_{t}^{n},y_{t}^{q},z_{t}^{q},u_{t}^{n}\right)
-f\left(  t,x_{t}^{q},y_{t}^{q},z_{t}^{q},u_{t}^{n}\right)  \right]  dt\\
& -\left[  f\left(  t,x_{t}^{q},y_{t}^{q},z_{t}^{q},u_{t}^{n}\right)  -%
{\displaystyle\int\nolimits_{U}}
f\left(  t,x_{t}^{q},y_{t}^{q},z_{t}^{q},a\right)  q_{t}\left(  da\right)
\right]  dt\\
& +\left(  z_{t}^{n}-z_{t}^{q}\right)  dW_{t},\\
y_{T}^{n}-y_{T}^{q}= & \varphi\left(  x_{T}^{n}\right)  -\varphi\left(
x_{T}^{q}\right)  .
\end{array}
\right.
\]

Put%
\begin{align*}
Y_{t}^{n}  &  =y_{t}^{n}-y_{t}^{q},\\
Z_{t}^{n}  &  =z_{t}^{n}-z_{t}^{q},
\end{align*}
and%
\begin{align}
\Psi^{n}\left(  t,Y_{t}^{n},Z_{t}^{n}\right)   &  =-\left[  f\left(
t,x_{t}^{n},y_{t}^{q},z_{t}^{q},u_{t}^{n}\right)  -f\left(  t,x_{t}^{q}%
,y_{t}^{q},z_{t}^{q},u_{t}^{n}\right)  \right]  dt\\
&  -f\left(  t,x_{t}^{q},y_{t}^{q},z_{t}^{q},u_{t}^{n}\right)  -%
{\displaystyle\int\nolimits_{U}}
f\left(  t,x_{t}^{q},y_{t}^{q},z_{t}^{q},a\right)  q_{t}\left(  da\right)
\nonumber\\
&  -%
{\displaystyle\int\nolimits_{0}^{1}}
f_{y}\left(  t,x_{t}^{q},y_{t}^{q}+\lambda\left(  y_{t}^{n}-y_{t}^{q}\right)
,z_{t}^{q}+\lambda\left(  z_{t}^{n}-z_{t}^{q}\right)  ,u_{t}^{n}\right)
Y_{t}^{n}d\lambda\nonumber\\
&  -%
{\displaystyle\int\nolimits_{0}^{1}}
f_{z}\left(  t,x_{t}^{q},y_{t}^{q}+\lambda\left(  y_{t}^{n}-y_{t}^{q}\right)
,z_{t}^{q}+\lambda\left(  z_{t}^{n}-z_{t}^{q}\right)  ,u_{t}^{n}\right)
Z_{t}^{n}d\lambda.\nonumber
\end{align}

Then%
\begin{equation}
\left\{
\begin{array}
[c]{l}%
dY_{t}^{n}=\Psi^{n}\left(  t,Y_{t}^{n},Z_{t}^{n}\right)  dt+Z_{t}^{n}dW_{t},\\
Y_{T}^{n}=\varphi\left(  x_{T}^{n}\right)  -\varphi\left(  x_{T}^{q}\right)  .
\end{array}
\right.
\end{equation}

The above equation is a linear BSDE with bounded coefficients, then by
applying a priori estimates (see Briand et al $\left[  12\right]  $), we get%
\begin{align*}
\mathbb{E}\left[  \underset{t\in\left[  0,T\right]  }{\sup}\left\vert
Y_{t}^{n}\right\vert ^{2}+%
{\displaystyle\int\nolimits_{0}^{T}}
\left\vert Z_{t}^{n}\right\vert ^{2}dt\right]   &  \leq C\mathbb{E}\left[
\left\vert \varphi\left(  x_{T}^{n}\right)  -\varphi\left(  x_{T}^{q}\right)
\right\vert ^{2}+\left\vert
{\displaystyle\int\nolimits_{0}^{T}}
\left\vert \Psi^{n}\left(  t,0,0\right)  \right\vert dt\right\vert ^{2}\right]
\\
&  \leq C\mathbb{E}\left[  \left\vert \varphi\left(  x_{T}^{n}\right)
-\varphi\left(  x_{T}^{q}\right)  \right\vert ^{2}+%
{\displaystyle\int\nolimits_{0}^{T}}
\left\vert \Psi^{n}\left(  t,0,0\right)  \right\vert ^{2}dt\right]  .
\end{align*}

From $\left(  42\right)  $, we get%
\begin{align*}
\mathbb{E}\left[  \underset{t\in\left[  0,T\right]  }{\sup}\left\vert
Y_{t}^{n}\right\vert ^{2}+%
{\displaystyle\int\nolimits_{0}^{T}}
\left\vert Z_{t}^{n}\right\vert ^{2}dt\right]   &  \leq C\mathbb{E}\left\vert
\varphi\left(  x_{T}^{n}\right)  -\varphi\left(  x_{T}^{q}\right)  \right\vert
^{2}\\
&  +C\mathbb{E}%
{\displaystyle\int\nolimits_{0}^{T}}
\left\vert f\left(  t,x_{t}^{n},y_{t}^{q},z_{t}^{q},u_{t}^{n}\right)
-f\left(  t,x_{t}^{q},y_{t}^{q},z_{t}^{q},u_{t}^{n}\right)  \right\vert
^{2}dt\\
&  +C\mathbb{E}%
{\displaystyle\int\nolimits_{0}^{T}}
\left\vert f\left(  t,x_{t}^{q},y_{t}^{q},z_{t}^{q},u_{t}^{n}\right)  -%
{\displaystyle\int\nolimits_{U}}
f\left(  t,x_{t}^{q},y_{t}^{q},z_{t}^{q},a\right)  q_{t}\left(  da\right)
\right\vert ^{2}dt.
\end{align*}

By $\left(  4\right)  $, $\varphi$ and $f$ are uniformly Lipshitz with respect
to $x$, then we get%
\begin{align}
\mathbb{E}\left[  \underset{t\in\left[  0,T\right]  }{\sup}\left\vert
Y_{t}^{n}\right\vert ^{2}+%
{\displaystyle\int\nolimits_{0}^{T}}
\left\vert Z_{t}^{n}\right\vert ^{2}dt\right]   &  \leq C\mathbb{E}\left\vert
x_{T}^{n}-x_{T}^{q}\right\vert ^{2}+C\mathbb{E}%
{\displaystyle\int\nolimits_{0}^{T}}
\left\vert x_{t}^{n}-x_{t}^{q}\right\vert ^{2}dt\nonumber\\
&  +C\mathbb{E}%
{\displaystyle\int\nolimits_{0}^{T}}
\left\vert
{\displaystyle\int\nolimits_{U}}
f\left(  t,x_{t}^{q},y_{t}^{q},z_{t}^{q},a\right)  \delta_{u_{t}^{n}}\left(
da\right)  -%
{\displaystyle\int\nolimits_{U}}
f\left(  t,x_{t}^{q},y_{t}^{q},z_{t}^{q},a\right)  q_{t}\left(  da\right)
\right\vert ^{2}dt.\nonumber
\end{align}

By $\left(  38\right)  $, the first and second terms in the right hand side of
the above inequality tends to zero as $n$ tends to infinity. Moreover, since
$f$ is bounded, measurable and continuous with respect to $a$, then by
$\left(  37\right)  $ and the dominated convergence theorem, the third term in
the right hand side tends to zero as $n$ tends to infinity. This prove
$\left(  39\right)  $ and $\left(  40\right)  $.

\ 

iii) Proof of $\left(  41\right)  .$

Since $g$\ ,$h$\ and $l$ are uniformly Lipshitz with respect to $\left(
x,y,z\right)  $, then by using the Cauchy-Schwartz inequality, we have%
\begin{align*}
&  \left\vert J\left(  q^{n}\right)  -\mathcal{J}\left(  q\right)  \right\vert
\\
&  \leq C\left(  \mathbb{E}\left\vert x_{T}^{n}-x_{T}^{q}\right\vert
^{2}\right)  ^{1/2}+C\left(  \mathbb{E}\left\vert y_{0}^{n}-y_{0}%
^{q}\right\vert ^{2}\right)  ^{1/2}\\
&  +C\left(  \int_{0}^{T}\mathbb{E}\left\vert x_{t}^{n}-x_{t}^{q}\right\vert
^{2}ds\right)  ^{1/2}+C\left(  \int_{0}^{T}\mathbb{E}\left\vert y_{t}%
^{n}-y_{t}^{q}\right\vert ^{2}ds\right)  ^{1/2}+C\left(  \mathbb{E}\int
_{0}^{T}\left\vert z_{t}^{n}-z_{t}^{q}\right\vert ^{2}dt\right)  ^{1/2}\\
&  +\left(  \mathbb{E}\int_{0}^{T}\left\vert \int_{U}l\left(  t,x_{t}%
^{q},y_{t}^{q},z_{t}^{q},a\right)  \delta_{u_{t}^{n}}\left(  da\right)
dt-\int_{U}l\left(  t,x_{t}^{q},y_{t}^{q},z_{t}^{q},a\right)  q_{t}\left(
da\right)  \right\vert ^{2}dt\right)  ^{1/2}.
\end{align*}

By $\left(  38\right)  $, $\left(  39\right)  $ and $\left(  40\right)  $\ the
first five terms in the right hand side converge to zero. Furthermore, since
$h$ is bounded, measurable and continuous in $a$, then by $\left(  37\right)
$ and the dominated convergence theorem, the sixth term in the right hand side
tends to zero as $n$ tends to infinity. This prove $\left(  41\right)  $.
\end{proof}

\begin{lemma}
As a consequence of $\left(  41\right)  $, the strict and the relaxed control
problems have the same value functions. That is
\begin{equation}
\underset{v\in\mathcal{U}}{\inf}J\left(  v\right)  =\underset{q\in\mathcal{R}%
}{\inf}\mathcal{J}\left(  q\right)  .
\end{equation}

\end{lemma}

\begin{proof}
Let $u\in\mathcal{U}$ and $\mu\in\mathcal{R}$ be respectively a strict and
relaxed controls such that%
\begin{align}
J\left(  u\right)   &  =\underset{v\in\mathcal{U}}{\inf}J\left(  v\right) \\
\mathcal{J}\left(  \mu\right)   &  =\underset{q\in\mathcal{R}}{\inf
}\mathcal{J}\left(  q\right)  .
\end{align}

By $\left(  46\right)  $, we have
\[
\mathcal{J}\left(  \mu\right)  \leq\mathcal{J}\left(  q\right)  \text{,
}\forall q\in\mathcal{R}\text{.}%
\]

Since $\delta\left(  \mathcal{U}\right)  \subset\mathcal{R}$, then%
\[
\mathcal{J}\left(  \mu\right)  \leq\mathcal{J}\left(  q\right)  \text{,
}\forall q\in\delta\left(  \mathcal{U}\right)  \text{.}%
\]

Since $q\in\delta\left(  \mathcal{U}\right)  $, then $q=\delta_{v}$, where
$v\in\mathcal{U}$.

Then we get%
\[
\left\{
\begin{array}
[c]{c}%
\left(  x^{q},y^{q},z^{q}\right)  =\left(  x^{v},y^{v},z^{v}\right)  ,\\
\mathcal{J}\left(  q\right)  =J\left(  v\right)  .
\end{array}
\right.
\]

Hence,%
\[
\mathcal{J}\left(  \mu\right)  \leq J\left(  v\right)  \text{, }\forall
v\in\mathcal{U}\text{.}%
\]

The control $u$ becomes an element of $\mathcal{U}$, then we get%
\begin{equation}
\mathcal{J}\left(  \mu\right)  \leq J\left(  u\right)  \text{.}%
\end{equation}

On the other hand, by $\left(  45\right)  $ we have%
\begin{equation}
J\left(  u\right)  \leq J\left(  v\right)  \text{, }\forall v\in
\mathcal{U}\text{.}%
\end{equation}

The control $\mu$ becomes a relaxed control, then by lemma $13$, there exists
a sequence $\left(  u^{n}\right)  _{n}$ of strict controls such
that\textit{\ }%
\[
dt\mu_{t}^{n}\left(  da\right)  =dt\delta_{u_{t}^{n}}\left(  da\right)
\underset{n\longrightarrow\infty}{\longrightarrow}dt\mu_{t}\left(  da\right)
\text{ stably},\text{\textit{\ \ }}\mathcal{P}-a.s.
\]

By $\left(  48\right)  $, we get then%
\[
J\left(  u\right)  \leq J\left(  u^{n}\right)  \text{, }\forall n\in
\mathbb{N}\text{,}%
\]

By using $\left(  41\right)  $ and letting $n$ go to infinity in\ the above
inequality, we get%
\begin{equation}
J\left(  u\right)  \leq\mathcal{J}\left(  \mu\right)  .
\end{equation}

Finally, by $\left(  47\right)  $ and $\left(  49\right)  $, the proof is completed.
\end{proof}

\ 

To establish necessary optimality conditions for strict controls, we need the
following lemma

\begin{lemma}
The strict control $u$ minimizes $J$ over $\mathcal{U}$ if and only if the
relaxed control $\mu=\delta_{u}$ minimizes $\mathcal{J}$ over $\mathcal{R}$.
\end{lemma}

\begin{proof}
Suppose that $u$ minimizes the cost $J$ over $\mathcal{U}$, then
\[
J\left(  u\right)  =\underset{v\in\mathcal{U}}{\inf}J\left(  v\right)
\text{.}%
\]

By using $\left(  44\right)  $, we get%
\[
J\left(  u\right)  =\underset{q\in\mathcal{R}}{\inf}\mathcal{J}\left(
q\right)  \text{.}%
\]

Since $\mu=\delta_{u}$, then%
\begin{equation}
\left\{
\begin{array}
[c]{c}%
\left(  x^{\mu},y^{\mu},z^{\mu}\right)  =\left(  x^{u},y^{u},z^{u}\right)  ,\\
\mathcal{J}\left(  \mu\right)  =J\left(  u\right)  ,
\end{array}
\right.
\end{equation}

This implies that%
\[
\mathcal{J}\left(  \mu\right)  =\underset{q\in\mathcal{R}}{\inf\mathcal{J}%
\left(  q\right)  }.
\]

Conversely, if $\mu=\delta_{u}$ minimize $\mathcal{J}$ over $\mathcal{R}$,
then%
\[
\mathcal{J}\left(  \mu\right)  =\underset{q\in\mathcal{R}}{\inf\mathcal{J}%
\left(  q\right)  }.
\]

From $\left(  44\right)  $, we get%
\[
\mathcal{J}\left(  \mu\right)  =\underset{v\in\mathcal{U}}{\inf J\left(
v\right)  }.
\]

Since $\mu=\delta_{u}$, then relations $\left(  50\right)  $ hold, and we
obtain%
\[
J\left(  u\right)  =\underset{v\in\mathcal{U}}{\inf J\left(  v\right)  }.
\]

The proof is completed.
\end{proof}

\ 

The following lemma, who will be used to establish sufficient optimality
conditions for strict controls, shows that we get the results of the above
lemma if we replace $\mathcal{R}$ by $\mathbb{\delta}\left(  \mathcal{U}%
\right)  .$

\begin{lemma}
The strict control $u$ minimizes $J$ over $\mathcal{U}$ if and only if the
relaxed control $\mu=\delta_{u}$ minimizes $\mathcal{J}$ over $\delta\left(
\mathcal{U}\right)  $.
\end{lemma}

\begin{proof}
Let $\mu=\delta_{u}$ be an optimal relaxed control minimizing the cost
$\mathcal{J}$ over $\delta\left(  \mathcal{U}\right)  $, we have then%
\[
\mathcal{J}\left(  \mu\right)  \leq\mathcal{J}\left(  q\right)  \text{,\ \ }%
\forall q\in\delta\left(  \mathcal{U}\right)  .
\]

Since $q\in\delta\left(  \mathcal{U}\right)  $, then there exists
$v\in\mathcal{U}$ such that $q=\delta_{v}.$

It is easy to see that%
\begin{equation}
\left\{
\begin{array}
[c]{c}%
\left(  x^{\mu},y^{\mu},z^{\mu}\right)  =\left(  x^{u},y^{u},z^{u}\right)  ,\\
\left(  x^{q},y^{q},z^{q}\right)  =\left(  x^{v},y^{v},z^{v}\right)  ,\\
\mathcal{J}\left(  \mu\right)  =J\left(  u\right)  ,\\
\mathcal{J}\left(  q\right)  =J\left(  v\right)  .
\end{array}
\right.
\end{equation}

Then, we get%
\[
J\left(  u\right)  \leq J\left(  v\right)  ,\ \ \forall v\in\mathcal{U}%
\text{.}%
\]

Conversely, let $u$ be a strict control minimizing the cost $J$ over
$\mathcal{U}$. Then%
\[
J\left(  u\right)  \leq J\left(  v\right)  ,\ \ \forall v\in\mathcal{U}%
\text{.}%
\]

Since the controls $u,v$ $\in\mathcal{U}$, then there exist $\mu,q\in
\delta\left(  \mathcal{U}\right)  $ such that
\[
\mu=\delta_{u}\ \ \ ,\ \ \ q=\delta_{v}.
\]

This implies that relations $\left(  51\right)  $ hold. Consequently, we get%
\[
\mathcal{J}\left(  \mu\right)  \leq\mathcal{J}\left(  q\right)  \text{,\ \ }%
\forall q\in\delta\left(  \mathcal{U}\right)  .
\]

The lemma is proved.
\end{proof}

\subsection{Necessary optimality conditions for strict controls}

Define the Hamiltonian $H$\ in the strict case from $\left[  0,T\right]
\times\mathbb{R}^{n}\times\mathbb{R}^{m}\times\mathcal{M}_{m\times d}\left(
\mathbb{R}\right)  \times U\times\mathbb{R}^{m}\times\mathbb{R}^{n}%
\times\mathcal{M}_{n\times d}\left(  \mathbb{R}\right)  $\ into $\mathbb{R}%
$\ by%
\[
H\left(  t,x,y,z,v,k,p,P\right)  =l\left(  t,x,y,z,v\right)  +pb\left(
t,x,v\right)  +P\sigma\left(  t,x,v\right)  +kf\left(  t,x,y,z,v\right)  .
\]

\begin{theorem}
(Necessary optimality conditions for strict controls). \textit{Let }%
$u$\textit{\ be an optimal control minimizing the functional }$J$%
\textit{\ over }$\mathcal{U}$\textit{\ and }$\left(  x^{u},y^{u},z^{u}\right)
$\textit{\ the solution of }$\left(  1\right)  $\textit{\ associated with }%
$u$\textit{. }Then, there exist three adapted processes $\left(  p^{\mu
},P^{\mu},k^{\mu}\right)  $, unique solution of the following FBSDE
system\textit{ (called adjoint equations)}%
\begin{equation}
\left\{
\begin{array}
[c]{ll}%
dk_{t}^{u}= & H_{y}\left(  t,x_{t}^{u},y_{t}^{u},z_{t}^{u},u_{t},k_{t}%
^{u},p_{t}^{u},P_{t}^{u}\right)  dt\\
& +H_{z}\left(  t,x_{t}^{u},y_{t}^{u},z_{t}^{u},u_{t},k_{t}^{u},p_{t}%
^{u},P_{t}^{u}\right)  dW_{t},\\
k_{0}^{u}= & h_{y}\left(  y_{0}^{u}\right) \\
dp_{t}^{u}= & -H_{x}\left(  t,x_{t}^{u},y_{t}^{u},z_{t}^{u},u_{t},k_{t}%
^{u},p_{t}^{u},P_{t}^{u}\right)  dt+P_{t}^{u}dW_{t},\\
p_{T}^{u}= & g_{x}\left(  x_{T}^{u}\right)  +\varphi_{x}\left(  x_{T}%
^{u}\right)  k_{T}^{u},
\end{array}
\right.
\end{equation}
such that for every $v_{t}\in U$%
\begin{equation}
H\left(  t,x_{t}^{u},y_{t}^{u},z_{t}^{u},u_{t},k_{t}^{u},p_{t}^{u},P_{t}%
^{u}\right)  \leq H\left(  t,x_{t}^{u},y_{t}^{u},z_{t}^{u},v_{t},k_{t}%
^{u},p_{t}^{u},P_{t}^{u}\right)  ,\ ae\ ,\ as.
\end{equation}

\end{theorem}

\begin{proof}
Let $u$ be an optimal solution of the strict control problem $\left\{  \left(
1\right)  ,\left(  2\right)  ,\left(  3\right)  \right\}  $. Then, there exist
$\mu\in\delta\left(  \mathcal{U}\right)  $ such that
\[
\mu=\delta_{u}.
\]

Since $u$ minimizes the cost $J$ over $\mathcal{U}$, then by lemma $17$, $\mu$
minimizes $\mathcal{J}$ over $\mathcal{R}$. Hence, by the necessary optimality
conditions for relaxed controls (Theorem $11$), there exist three unique
adapted processes $\left(  k^{\mu},p^{\mu},P^{\mu}\right)  $, solution of the
system of relaxed adjoint equations $\left(  30\right)  $ such that, for every
$q_{t}\in\mathbb{P}\left(  U\right)  $%
\[
\mathcal{H}\left(  t,x_{t}^{\mu},y_{t}^{\mu},z_{t}^{\mu},\mu_{t},k_{t}^{\mu
},p_{t}^{\mu},P_{t}^{\mu}\right)  \leq\mathcal{H}\left(  t,x_{t}^{\mu}%
,y_{t}^{\mu},z_{t}^{\mu},q_{t},k_{t}^{\mu},p_{t}^{\mu},P_{t}^{\mu}\right)
,\ a.e,\ a.s.
\]

Since $\mathbb{\delta}\left(  U\right)  \subset\mathbb{P}\left(  U\right)  $,
then for every $v_{t}\in\delta\left(  U\right)  $, we get
\begin{equation}
\mathcal{H}\left(  t,x_{t}^{\mu},y_{t}^{\mu},z_{t}^{\mu},\mu_{t},k_{t}^{\mu
},p_{t}^{\mu},P_{t}^{\mu}\right)  \leq\mathcal{H}\left(  t,x_{t}^{\mu}%
,y_{t}^{\mu},z_{t}^{\mu},q_{t},k_{t}^{\mu},p_{t}^{\mu},P_{t}^{\mu}\right)
,\ a.e,\ a.s.
\end{equation}

Since $q\in\delta\left(  \mathcal{U}\right)  $, then there exist
$v\in\mathcal{U}$ such that $q=\delta_{v}$.

We note that $v$ is an arbitrary element of $\mathcal{U}$ since $q$ is arbitrary.

Now, since $\mu=\delta_{u}$ and $q=\delta_{v}$, we can easily see that%
\begin{equation}
\left\{
\begin{array}
[c]{c}%
\left(  x^{\mu},y^{\mu},z^{\mu}\right)  =\left(  x^{u},y^{u},z^{u}\right)  ,\\
\left(  x^{q},y^{q},z^{q}\right)  =\left(  x^{v},y^{v},z^{v}\right)  ,\\
\left(  k^{\mu},p^{\mu},P^{\mu}\right)  =\left(  k^{u},p^{u},P^{u}\right)  ,\\
\mathcal{H}\left(  t,x_{t}^{\mu},y_{t}^{\mu},z_{t}^{\mu},\mu_{t},k_{t}^{\mu
},p_{t}^{\mu},P_{t}^{\mu}\right)  =H\left(  t,x_{t}^{u},y_{t}^{u},z_{t}%
^{u},u_{t},k_{t}^{u},p_{t}^{u},P_{t}^{u}\right)  ,\\
\mathcal{H}\left(  t,x_{t}^{\mu},y_{t}^{\mu},z_{t}^{\mu},q_{t},k_{t}^{\mu
},p_{t}^{\mu},P_{t}^{\mu}\right)  =H\left(  t,x_{t}^{u},y_{t}^{u},z_{t}%
^{u},v_{t},k_{t}^{u},p_{t}^{u},P_{t}^{u}\right)  ,
\end{array}
\right.
\end{equation}
where, the pair $\left(  p^{u},P^{u}\right)  $ and $k^{u}$ are respectively
the unique solutions of the system of strict adjoint equations $\left(
52\right)  $.

Finally, by using $\left(  54\right)  $ and $\left(  55\right)  $, we can easy
deduce $\left(  53\right)  $. The proof is completed.
\end{proof}

\subsection{Sufficient optimality conditions for strict controls}

\begin{theorem}
(Sufficient optimality conditions for strict controls). Assume that the
functions $g$, and $\left(  x,y,z\right)  \longmapsto H\left(
t,x,y,z,q,k,p,P\right)  $ are convex, and for any $v\in\mathcal{U}$,
$y_{T}^{v}=\xi,$ where $\xi$\ is an $m$-dimensional $\mathcal{F}_{T}%
$-measurable random variable such that%
\[
\mathbb{E}\left\vert \xi\right\vert ^{2}<\infty.
\]

Then, $u$ is an optimal solution of the control problem $\left\{  \left(
1\right)  ,\left(  2\right)  ,\left(  3\right)  \right\}  $, if it satisfies
$\left(  53\right)  .$
\end{theorem}

\begin{proof}
Let $u$ be a strict control (candidate to be optimal) such that necessary
optimality conditions for strict controls\ $\left(  53\right)  $ hold. i.e,
for every $v_{t}\in U$%
\begin{equation}
H\left(  t,x_{t}^{u},y_{t}^{u},z_{t}^{u},u_{t},k_{t}^{u},p_{t}^{u},P_{t}%
^{u}\right)  \leq H\left(  t,x_{t}^{u},y_{t}^{u},z_{t}^{u},v_{t},k_{t}%
^{u},p_{t}^{u},P_{t}^{u}\right)  ,\ a.e,\ a.s.
\end{equation}

The controls $u,v$ are elements of $\mathcal{U}$, then there exist $\mu
,q\in\delta\left(  \mathcal{U}\right)  $ such that%
\begin{align*}
\mu &  =\delta_{u},\\
q  &  =\delta_{v}.
\end{align*}

This implies that relations $\left(  55\right)  $ hold. Then by $\left(
56\right)  $, we deduce that for every $q_{t}\in\mathbb{\delta}\left(
U\right)  $%
\[
\mathcal{H}\left(  t,x_{t}^{\mu},y_{t}^{\mu},z_{t}^{\mu},\mu_{t},k_{t}^{\mu
},p_{t}^{\mu},P_{t}^{\mu}\right)  \leq\mathcal{H}\left(  t,x_{t}^{\mu}%
,y_{t}^{\mu},z_{t}^{\mu},q_{t},k_{t}^{\mu},p_{t}^{\mu},P_{t}^{\mu}\right)
,\ a.e,\ a.s.
\]

Since $H$ is convex in $\left(  x,y,z\right)  $, it is easy to see that
$\mathcal{H}$ is convex in $\left(  x,y,z\right)  $, and since $g$ and $h$ are
convex, then by the same proof that in theorem $12$, we show that $\mu$
minimizes the cost $\mathcal{J}$ over $\mathbb{\delta}\left(  \mathcal{U}%
\right)  $. Finally by lemma $18$, we deduce that $u$ minimizes the cost $J$
over $\mathcal{U}$. The theorem is proved.
\end{proof}

\begin{remark}
The sufficient optimality conditions for strict controls are proved without
assuming neither the convexity of $U$ nor that of $H$ in $v$.
\end{remark}

\end{document}